\documentclass[11pt,reqno]{amsart}

\usepackage{bm}
\usepackage{bbm}
\usepackage{mathrsfs}
\usepackage{amsmath, amsthm, amsfonts, amssymb, amsbsy, graphicx, enumerate, caption}
\usepackage{array}
\usepackage{makecell,multirow,tabularx,ragged2e}
\usepackage{cellspace}
\usepackage{booktabs}
\usepackage{tikz}
\usepackage{times}
\usepackage{hyperref}
\usepackage{dsfont}
\usepackage{enumitem}
\usepackage[margin=1in]{geometry}
\usepackage[foot]{amsaddr}
\usepackage{color}
\usepackage{parskip}
\usepackage{mathrsfs}
\usepackage{stmaryrd}
\usepackage{cases}
\usepackage[T1]{fontenc}

\newtheorem{theorem}{Theorem}[section]
\newtheorem{proposition}{Proposition}[section]
\newtheorem{lemma}{Lemma}[section]

\newtheorem{definition}{Definition}[section]

\newtheorem{remark}{Remark}[section]

\newcommand{\R}{\mathbb{R}}
\newcommand{\N}{\mathbb{N}}

\newcommand{\Z}{\mathbb{Z}}
\renewcommand{\P}{\mathbb{P}}
\newcommand{\E}{\mathbb{E}}
\newcommand{\1}{\mathbf{1}}

\newcommand{\eps}{\epsilon}

\newcommand{\bdy}{\partial}
\newcommand{\dist}{\mathrm{dist}}

\newcommand{\diam}{\mathrm{diam}}
\newcommand{\wt}{\widehat}

\newcommand{\llb}{\llbracket}
\newcommand{\rrb}{\rrbracket}

\newcommand{\Prog}{\mathrm{Prog}}
\newcommand{\wtProg}{\mathrm{\widehat{P}rog}}

\newcommand{\T}{\mathcal{T}}
\renewcommand{\AA}{\mathcal{A}}
\newcommand{\C}{\mathscr{C}}
\newcommand{\D}{\mathscr{D}}

\newcommand{\Part}{\mathrm{Part}}

\newcommand{\sep}{\mathrm{sep}}

\renewcommand{\dim}{d}
\newcommand{\num}{n}

\newcommand{\kk}{m}
\newcommand{\nn}{n}
\newcommand{\dd}{d}
\newcommand{\X}{\mathscr{X}}

\newcommand{\Tup}{\mathrm{Tup}}
\newcommand{\comp}{\mathrm{comp}}
\newcommand{\cal}{\mathcal}
\newcommand{\scr}{\mathscr}
\newcommand{\Stay}{\mathsf{Stay}}
\newcommand{\Stable}{\mathsf{CanPersist}}
\newcommand{\NoSmallComps}{\mathsf{NoSmallComp}}
\newcommand{\TreadPart}{\mathsf{Tread}}
\newcommand{\AddToBall}{\mathsf{AddToBall}}
\newcommand{\DepleteSmallComps}{\mathsf{DepleteSmallComps}}

\newcommand{\Growth}{\mathsf{Growth}}

\title{Critical numerosity in collective behavior}
\author{Jacob Calvert}
\email{jacob\_calvert@berkeley.edu}


\begin{document}

\begin{abstract}
Natural collectives, despite comprising individuals who may not know their numerosity, can exhibit behaviors that depend sensitively on it. This paper proves that the collective behavior of number-oblivious individuals can even have a {\em critical numerosity}, above and below which it qualitatively differs. We formalize the concept of critical numerosity in terms of a family of zero--one laws and introduce a model of collective motion, called {\em chain activation and transport} (CAT), that has one.

CAT describes the collective motion of $n \geq 2$ individuals as a Markov chain that rearranges $n$-element subsets of the $d$-dimensional grid, $m < n$ elements at a time. According to the individuals' dynamics, with each step, CAT removes $m$ elements from the set and then progressively adds $m$ elements to the boundary of what remains, in a way that favors the consecutive addition and removal of nearby elements. This paper proves that, if $d \geq 3$, then CAT has a critical numerosity of $n_c = 2m+2$ with respect to the behavior of its diameter. Specifically, if $n < n_c$, then the elements form one ``cluster,'' the diameter of which has an a.s.--finite limit infimum. However, if $n \geq n_c$, then there is an a.s.--finite time at which the set consists of clusters of between $m+1$ and $2m+1$ elements, and forever after which these clusters grow apart, resulting in unchecked diameter growth.

The existence of critical numerosities means that collectives can exhibit ``phase transitions'' that are governed purely by their numerosity and not, for example, their density or the strength of their interactions. This fact challenges prevalent beliefs about collective behavior; inspires basic scientific questions about the role of numerosity in the behavior of natural collectives; and suggests new functionality for programmable matter, like robot swarms, smart materials, and synthetic biological systems. More broadly, it demonstrates an opportunity to explore the possible behaviors of natural and engineered collectives through the study of random processes that rearrange finite sets.
\end{abstract}

\maketitle

\setcounter{tocdepth}{1}
\tableofcontents

\section{Introduction}

A fundamental question of any collective behavior is: How numerous must the collective be to exhibit this behavior? In particular, is there a {\em critical numerosity}\footnote{We use ``numerosity'' because it is more descriptive than ``number'' and, unlike ``size,'' cannot be confused with the spatial extent of the collective.} above which the collective exhibits this behavior but below which it does not? Experimental evidence suggests that natural collectives might exhibit critical numerosities with respect to important behaviors (Table~\ref{exp summary}). 
This possibility is remarkable because the constituents of these collectives do not necessarily know their number, although they may be able to infer it \cite{pratt2005}. To explore these questions, this paper formalizes the concept of critical numerosity and introduces a model of collective motion that provably exhibits critical numerosities.

\renewcommand*{\thead}[1]{\multicolumn{1}{c}{\bfseries #1}}
\bgroup
\def\arraystretch{1.5}
\begin{table}[ht]
\centering
\caption{Six experiments that suggest the existence of critical numerosities ($n_c$). The precise definitions of behavior and suggested values of $n_c$ depend on experimental details.}
\begin{tabular}{ l l c c}
\toprule
\thead{Type of collective} & \thead{Behavior} & \makecell[cc]{{\bf Suggested $n_c$}} & \thead{Reference} \\ 
\midrule
\makecell*[tl]{Colony of Pharaoh's ants\\ ({\em Monomorium} {\em pharaonis})} & \makecell[tl]{Foraging is trail-based} & $\approx 700$ & \cite{beekman2001}\\ 
\makecell*[tl]{Swarm of midges\\ ({\em Chironomus riparius})} & \makecell[tl]{Volume-per-midge\\ is saturated} & $\approx 10$ & \cite{puckett2014}\\
\makecell*[tl]{Colony of leafcutter ants\\ ({\em Acromyrmex versicolor})} & \makecell[tl]{Population grows with\\ fungal cultivar mass} & $\approx 89$ & \cite{clark2014}\\
\makecell*[tl]{Cluster of water molecules} & Forms ice I & $\approx 90$ & \cite{moberg2019} \\ 
\makecell*[tl]{School of cichlid fish\\ ({\em Etroplus suratensis})} & \makecell[tl]{The likeliest state\\ of motion is isotropic} & $\approx 60$ & \cite{jhawar2020} \\
\makecell*[tl]{Colony of fire ants\\ ({\em Solenopsis invicta})} & \makecell[tl]{Forms a stable raft} & $\approx 10$ & \cite{ko2022} \\
\bottomrule
\end{tabular}
\label{exp summary}
\end{table}
\egroup 

The existence of critical numerosities means that collectives can exhibit phase transitions---in the sense of abrupt, qualitative changes---in behavior that are mediated {\em purely by the number of individuals} and not, for example, their density or the nature or strength of their interactions. This fact challenges conventional wisdom about collective behavior. For example, Ouellette and Gordon write that ``the properties and functionality of [a collective] arise from the interactions among the individuals. This means that no individual is essential for the group to function'' \cite{ouellette2021}. However, if the collective is at the critical numerosity for a behavior, then {\em every} individual is essential. For the same reason, the existence of critical numerosities suggests new functionality for programmable matter, like robot swarms. Indeed, an experimentally validated approach to programming matter uses a phase transition inspired by the Ising model; by increasing the strength of attractive interactions between individual robots, the swarm transitions from a dispersed phase to a compact phase that exhibits collective transport \cite{li2021}. Analogously, the existence of critical numerosities could enable programmable matter to exhibit different functionality depending on its numerosity.

\subsection{Notation}
{\em Concerning $\Z$ and $\R$}. For $i,j \in \Z$, we denote $\Z_{\geq i} = \{i, i+1, \dots\}$ and the integer interval $\llbracket i,j \rrbracket = \{k \in \Z: i \leq k \leq j\}$. In particular, we denote $\N = \Z_{\geq 0}$. We denote the nonnegative (positive) real numbers by $\R_{\geq 0}$ ($\R_{> 0}$). We use $[r]$ to denote the integer part of $r \in \R$.

{\em Concerning $\Z^d$}. For $d \in \Z_{\geq 1}$ and $A,B \subset \Z^d$, we define $\diam (A) := \sup_{x,y \in A} \|x-y\|$ and $\dist (A,B) := \inf_{x \in A, \, y \in B} \|x - y\|$, where $\|\cdot\|$ is the Euclidean norm. For $A \subset \Z^d$, we denote the boundary and closure of $A$ by 
$\bdy A := \{x \in \Z^d \setminus A: \dist(x,A) = 1\}$ and $\overline{A} = A \cup \bdy A$. 
We denote the discrete Euclidean ball of radius $r \in \R_{>0}$ about $x \in \Z^d$ by $\cal B_x (r) := \{y \in \Z^d: \|x - y\| < r\}$. For $i\leq j$, we denote by $L_{i,j}$ the line segment $\{i e_1, \dots, j e_1\}$, where $e_j$ denotes the $j$\textsuperscript{th} standard unit vector of $\R^d$, and we denote the collection of translates of such line segments by $\cal L := \big\{ L_{1,k} + x: k \in \Z_{\geq 1}, x \in \Z^d \big\}$. We define $\comp_A (x)$ to be the set of elements that are connected to $x \in \Z^d$ in $A \subseteq \Z^d$. In particular, $\comp_A (x) = \emptyset$ whenever $x \notin A$.

{\em Miscellaneous}. For $k \in \Z_{\geq 1}$, we use $A \subset_k B$ to denote that $A$ is a $k$-element subset of $B$. Additionally, we use $f \lesssim g$ to denote the estimate $f \leq c g$ for a ``universal'' positive number $c$ and $O(g)$ to denote such a quantity $f$. We use $f \gtrsim g$ and $\Omega (g)$ analogously, for the reverse estimate. Lastly, for a random process $\scr X = (\scr X_t)_{t \in \N}$ on a state space $\cal S$, we use $\P_X$ ($\E_X$) to denote the conditional distribution (expectation) of $\scr X$, given that $\scr X_0 = X$ for $X \in \cal S$, and we use $\sigma (\scr X)$ to denote the $\sigma$-field that $\scr X$ generates.

\subsection{Definition of critical numerosity}

We model the state of an individual as an element of $\Z^d$ and the state of the collective as an element of $\cal S \subseteq \{S \subset \Z^d: 2 \leq |S| < \infty\}$. The dynamics of the collective is a random process $\X = (\X_0, \X_1, \dots)$ on $\cal S$, which conserves the number of individuals: 
\begin{equation}\label{mass conservation}
\forall X \in \cal S, \quad \P_{X} ( |\X_0| = |\X_1| = \cdots ) = 1.
\end{equation}
We define critical numerosity in terms of a family of zero--one laws (Figure~\ref{nc figure}).

\begin{definition}[Critical numerosity]\label{def nc} We say that $\X$ has a critical numerosity of $n_c \in \Z_{\geq 2}$ with respect to $B \in \sigma (\X)$, if
\begin{equation}\label{nc cases}
\forall X \in \cal S, \quad \P_{X} (B) =
\begin{cases}
0 & |X| < n_c,\\ 
1 & |X| \geq n_c.
\end{cases}
\end{equation}
\end{definition}

\begin{figure}
\includegraphics[width=0.45\linewidth]{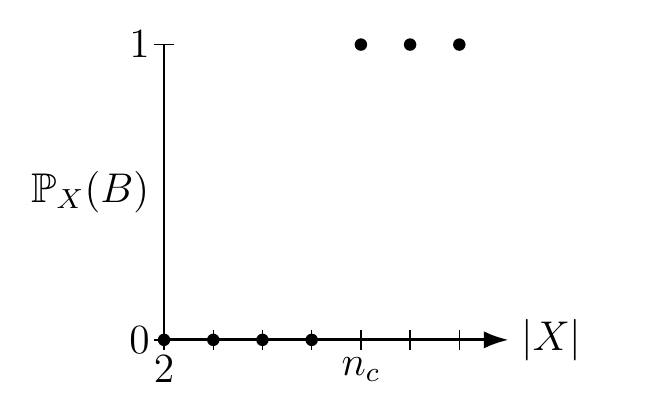}
\caption{A random process $\scr X = (\scr X_0, \scr X_1, \dots)$ on the state space $\cal S$ that satisfies \eqref{mass conservation} exhibits a critical numerosity of $n_c$ with respect to an event $B \in \sigma (\X)$ if this plot holds for every $X \in \cal S$.
}
\label{nc figure}
\end{figure}

Definition~\ref{def nc} aims to formalize the notion of a behavior $B$ that the collective modeled by $\scr X$ exhibits, if and only if it has $n_c$ or more constituents. While this definition suffices for our purposes, we note that its scope can be extended beyond discrete-time random processes on subsets of $\Z^d$, and there are alternative ways to formalize ``behavior'' and ``exhibits.'' For example, because a behavior $B$ can be any event in $\sigma (\X)$, Definition~\ref{def nc} permits contrived behaviors like $\{|\X_0| \geq n_c\}$ that satisfy \eqref{nc cases} through explicit reference to numerosity.\footnote{Similarly, because there is no restriction on the dynamics of $\X$, a contrived dynamics, like one that is nontrivial if and only if $|\X_0| \geq n_c$, can satisfy \eqref{nc cases} even if the behavior in question makes no reference to numerosity.} Disallowing behaviors in $\sigma (\cup_{t \in \N} |\X_t|)$ is an alternative. Concerning ``exhibits,'' an alternative replaces the condition that $\P_{X} (B) = 1$ when $|X| \geq n_c$ with the weaker condition that $\P_{X} (B) > 0$ when $|X| \geq n_c$.\footnote{While these conditions are generally inequivalent, if $\scr X$ is a recurrent Markov chain on $\cal S$ and if $B$ belongs to the tail $\sigma$-field of $\scr X$, then the zero--one law of Blackwell and Freedman \cite[Theorem 1]{blackwell1964} implies that they are equivalent.}

\subsection{Chain activation and transport}
The main results of this paper concern a model of collective motion, called {\em chain activation and transport} (CAT), which belongs to the class of Markov chains that rearrange finite subsets of a graph, like harmonic activation and transport \cite{calvert2021a,calvert2021b}, competitive erosion \cite{ganguly2017,ganguly2018,ganguly2019}, and the geometric amoebot model of programmable matter \cite{derakshandeh2014,cannon2016,arroyo2017,savoie2018,cannon2019,li2021}. In particular, CAT generalizes the Markov chain on the state space $\{S \subset \Z^d: 2 \leq |S| < \infty\}$ that, with each step, removes a uniformly random element of the set and then adds a uniformly random element to the boundary of what remains. 
The data underlying CAT are:
\begin{itemize}
\item $d \geq 1$, the dimension of the underlying grid;
\item $m \geq 1$, the number of elements that move with each step; and
\item $\beta \in \R$, a parameter that controls how far elements tend to move.
\end{itemize}

CAT is a Markov chain $(\C_t)_{t \geq 0}$ on the state space $\cal S = \{S \subset \Z^d: m+1 \leq |S| < \infty \}$ that, with each step, removes ({\em activates}) $m$ elements from the set and then adds ({\em transports}) $m$ elements to the boundary of the set, one element at a time, with a probability that depends on the distance between consecutive elements. In other words, there are $2m$ {\em substeps} and, denoting by $\C_{t,l}$ the set that results from $l$ substeps at time $t$, we have $\C_{t+1} = \C_{t,2m}$ (Figure~\ref{cat dyn fig}).

\begin{figure}
\includegraphics[width=0.7\linewidth]{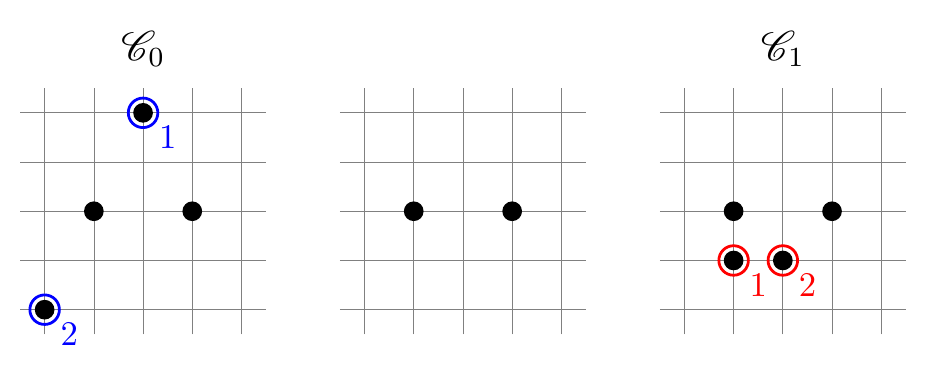}
\caption{An example of the dynamics with $(d,m,n) = (2,2,4)$. Activation occurs at the elements of $\C_0$ circled in blue (left), resulting in an intermediate set (middle). Transport occurs to the elements of $\C_1 = \C_{0,4}$ circled in red (right).}
\label{cat dyn fig}
\end{figure} 

We activate the first element uniformly at random in $\C_t$. Then, given that $x_l$ is the $l$\textsuperscript{th} element activated or transported, the $(l+1)$\textsuperscript{st} substep proceeds in one of two ways. If $l < m$, then we activate
\[
\text{$x_{l+1} \in \C_{t,l}$ with probability $\propto \| x_l - x_{l+1} \|^{-\beta}$.}
\] 
Otherwise, if $l \geq m$, then we transport to
\[
\text{$x_{l+1} \in \bdy \C_{t,l}$ with probability $\propto \| x_l - x_{l+1} \|^{-\beta}$.}
\]
In fact, to address the possibility that $x_l = x_{l+1}$ when $l=m$, we use $\phi (x_l,x_{l+1})$ in the place of $\|x_l- x_{l+1}\|^{-\beta}$, where
\begin{equation}\label{d def}
\phi (u,v) := \max\{\| u - v \|,1\}^{-\beta}
\end{equation}
for $u, v \in \Z^d$. Section~\ref{tp formula} provides a formula for the CAT transition probabilities.

We emphasize that, although activation begins uniformly at random under CAT, this does not mean that  individuals must know their numerosity for CAT to describe their collective motion. Indeed, we can think of CAT as describing the collective motion of number-oblivious individuals who initiate steps asynchronously, according to the times at which i.i.d.\ exponential clocks ring \cite{li2021}. Concerning the nonlocal components of the dynamics, we can think of CAT as describing individuals that respond to adaptive, long-range interactions like acoustic \cite{gorbonos2016,gorbonos2017,gorbonos2020b}, visual \cite{pearce2014,lavergne2019,bastien2020,harpaz2021}, or chemotactic \cite{daniels2004} interactions.

Four important properties of CAT are apparent from its definition.
\begin{enumerate}
\item[1.] {\em Conservation of numerosity}. CAT conserves the number of elements in the sense of \eqref{mass conservation}.
\item[2.] {\em Arbitrarily large diameter}. CAT configurations can have multiple connected components, hence the number of elements does not limit the diameter of the configuration. Specifically, if the configuration $C$ has at least three elements, then, for any $r>0$, the diameter of $\C_t$ eventually exceeds $r$, $\P_C$--a.s. 
\item[3.] {\em Asymmetric behavior of diameter}. Since elements must be transported to the boundary of other elements, the diameter of a configuration can increase by at most $m$ with each step. In other words, CAT has the {\em at most linear growth} (a.m.l.g.) property
\begin{equation}\label{amlg prop}
\P_C (\diam (\C_{t+1}) \leq \diam (\C_t) + m\,\,\text{for all $t \geq 0$}) = 1,
\end{equation}
for every configuration $C$. 
In contrast, the diameter can decrease abruptly. For example, if $n = m+1$ and $C \subset_n \Z^d$ has a diameter of $r \geq n$, then the diameter will decrease by at least $r - n$ in the first step, $\P_{C}$--a.s.
\item[4.] {\em Invariance under translation}. Since its transition probabilities depend only on the distances between elements, CAT satisfies
\[
\P_C ( \C_1 = D ) = \P_{C+x} ( \C_1 = D + x)
\]
for every $C, D \in \cal S$ and $x \in \Z^d$. We associate to each configuration $C$ the equivalence class consisting of its translates:
\begin{equation}\label{hat equiv class}
\wt C = \big\{ D \subset \Z^d: C = D + x \,\, \text{for some $x \in \Z^d$} \big\}.
\end{equation}
\end{enumerate}

\subsection{Statement of the main results}\label{main results}

In what follows, we assume that $d \geq 3$ and $\beta > 2$. Our main result is that CAT has a critical numerosity with respect to the behavior of its diameter (Figure~\ref{simFigure}). This critical numerosity can be made arbitrarily large through the choice of $m$.

\begin{figure}
\includegraphics[width=0.65\linewidth]{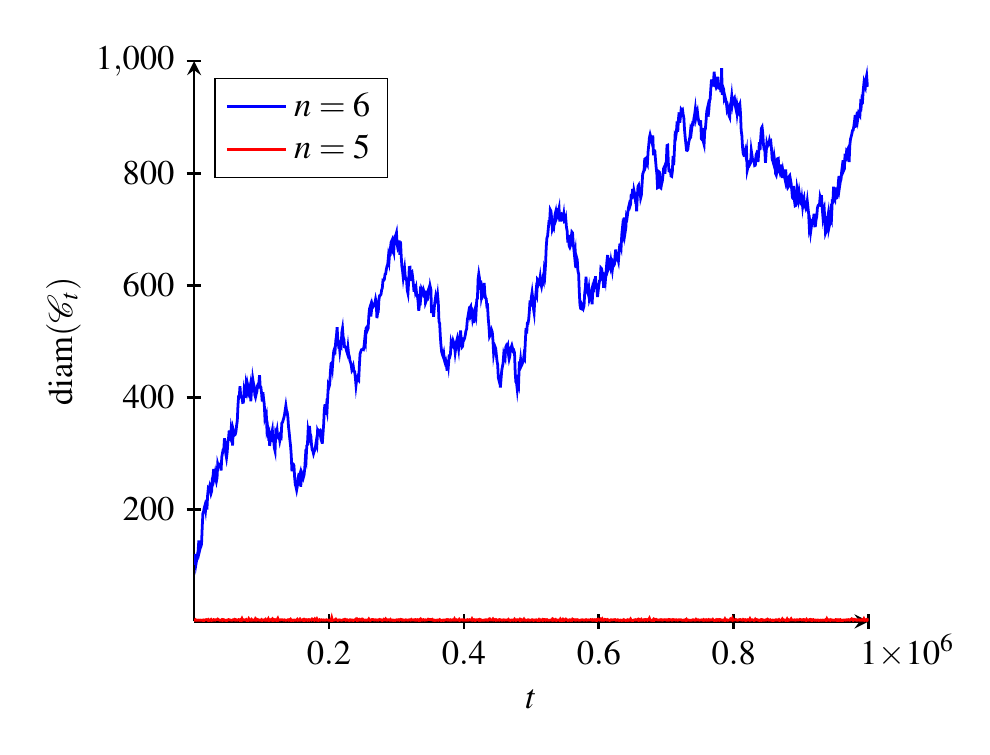}
\caption{Representative simulations of CAT with $(d,m,\beta) = (3,2,4)$ over $10^6$ steps, starting from $n$-element sets with diameters of $100$.}
\label{simFigure}
\end{figure}

\begin{theorem}[Critical numerosity]\label{phase transition}
CAT has a critical numerosity of $n_c = 2m+2$ with respect to the event $\{ \diam (\C_t) \to \infty\}$.
\end{theorem}

We can prove much more about the qualitative difference in the behaviors that CAT exhibits above and below $n_c$. This is the theme of the next three theorems. The first of these results states that, when $n < n_c$, the diameter of CAT exhibits {\em collapse} in the sense that, regardless of the diameter of the initial set, the diameter typically falls below a function of $n$ in $n$ steps.

\begin{theorem}[Diameter collapse]\label{drift}
If $n < 2m+2$, then there is $p > 0$ such that 
\begin{equation}\label{eq drift}
\P_{C} \left( \diam (\C_n) \leq 2n^2 \right) \geq p,
\end{equation}
for every $C \subset_n \Z^d$.
\end{theorem}

This result implies that if $n < n_c$, then the diameter of $\C_t$ is at most a function of $n$, infinitely often, a.s. Hence, the limit infimum of diameter is a.s.--finite in this case. The key feature of Theorem~\ref{drift} is that the number of steps, diameter upper bound, and probability lower bound in \eqref{eq drift} do not depend on the diameter of $C$; their values are otherwise unimportant. Indeed, if $p$ were allowed to depend on $\diam (C)$, then simple bounds on the transition probabilities would suffice to prove \eqref{eq drift}, without the hypothesis on $n$.

If $n \geq n_c$, then $\C_t$ eventually reaches a configuration consisting of well separated ``clusters'' that have the right number of elements to persist, in a sense. Specifically, the clusters have enough elements (at least $m+1$) to avoid being wholly activated in a single step, but not enough elements (at most $2m+1$) to split into smaller clusters that can persist.

To make this precise, we define $\Stable_{r}$ for $r \geq 0$ to be the subset of $\cal S$ consisting of sets that have a partition $(V^j)_{j=1}^J$ such that
\begin{equation}\label{reach well sep conds}
\forall j \in \llbracket 1, J \rrbracket, \quad V^j \in \cal L, \quad |V^j| \in \llbracket m+1, 2m+1 \rrbracket \quad \text{and} \quad \dist (V^j, V^{\neq j}) \geq r,
\end{equation}
where $V^{\neq j} = \cup_{i \neq j} V^i$. We stipulate that each $V^j$ belongs to $\cal L$, the collection of line segments parallel to $e_1$, to facilitate our use of the following result.

\begin{theorem}[Cluster formation]\label{reach well sep}
There is $p > 0$ such that  
\begin{equation}\label{eq reach well sep}
\P_{C} (\C_{2n^2 r} \in \Stable_r) \geq p^r,
\end{equation}
for every $C \subset_n \Z^d$ and $r \in \Z_{\geq 2n}$.
\end{theorem}

Theorem~\ref{reach well sep} states that it takes on the order of $r$ steps for CAT to reach a state consisting exclusively of $r$-separated clusters that can persist, with at least a probability depending on $r$. 
Like Theorem~\ref{drift}, the key feature of Theorem~\ref{reach well sep} is that the diameter of $C$ does not appear in \eqref{eq reach well sep}. If these clusters are sufficiently separated then, because they have the right number of elements to persist, they tend to move like $d$-dimensional random walks, hence they tend to grow in separation like $t^{1/2}$ over $t$ steps. As they grow increasingly separated, it becomes increasingly unlikely that the clusters exchange elements, favoring further separation growth, and so on. This mechanism underlies the growth of diameter when $n \geq n_c$. 

\begin{theorem}[Diameter growth]\label{growth} If $n \geq 2m+2$, then there is $\gamma > 0$ such that, for any $\delta \in (0,\frac12)$, there are $p>0$ and $s \in \Z_{\geq 0}$ for which
\begin{equation}\label{eq growth}
\P_{C} \left( \diam (\C_t) \geq \gamma (t-s)^{\frac12 - \delta} \,\,\text{for every $t \geq s$} \right) \geq p,
\end{equation}
for every $C \subset_n \Z^d$.
\end{theorem}

\begin{figure}
\includegraphics[width=0.75\linewidth]{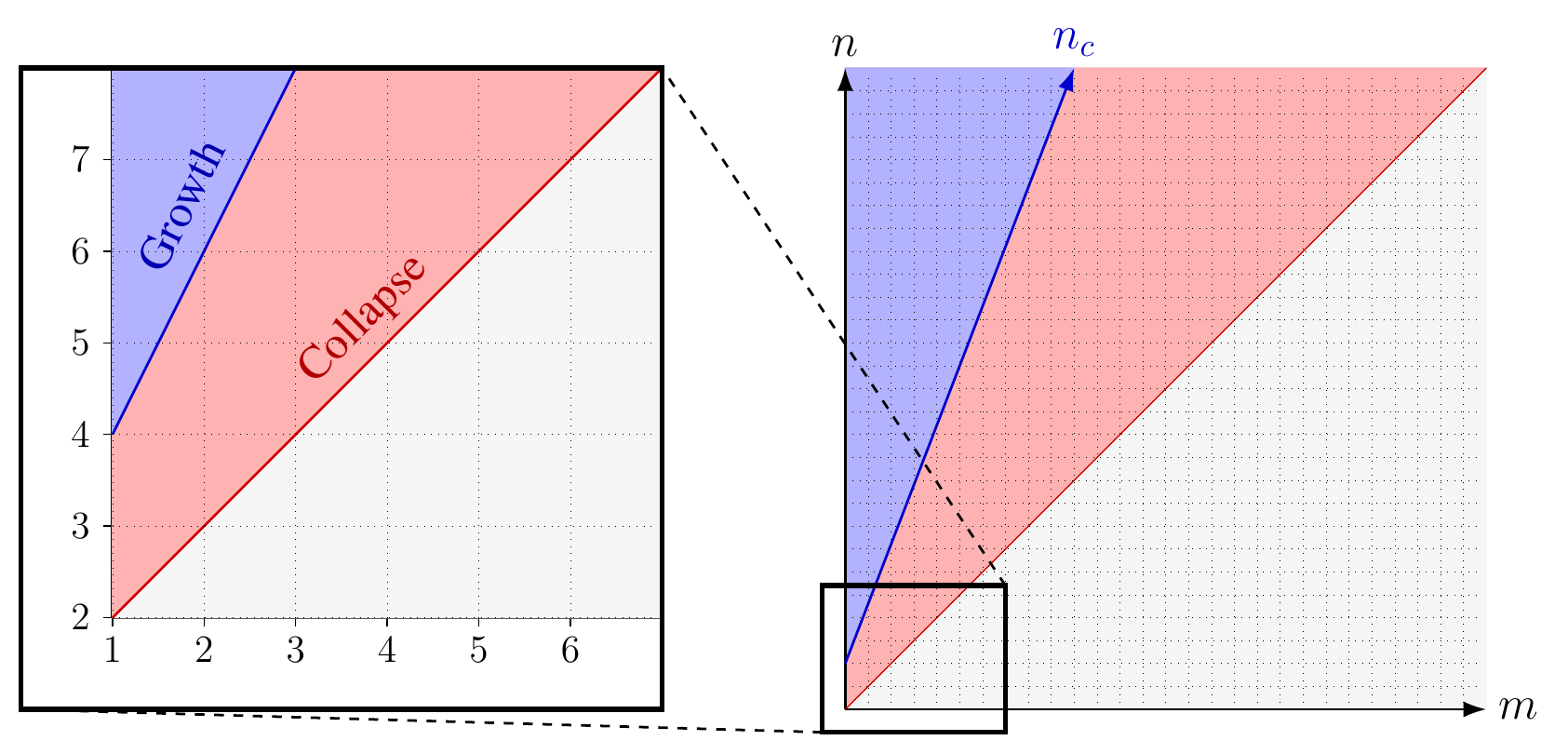}
\caption{The phase diagram of CAT in the $(m,n)$ grid includes regions of steady diameter growth (blue) and frequent diameter collapse (red), separated by critical numerosities (line labeled $n_c$). CAT is not defined for $n \leq m$ (gray).}
\label{phasediagram}
\end{figure}

Figure~\ref{phasediagram} summarizes Theorems~\ref{drift} and \ref{growth}, which imply Theorem~\ref{phase transition}.

\begin{proof}[Proof of Theorem~\ref{phase transition}]
Denote $B = \{\diam (\C_t) \to \infty\}$ and $n_c = 2m+2$. Theorem~\ref{drift} and the Borel--Cantelli lemma imply that, if $n < n_c$, then the limit infimum of $\diam(\C_t)$ is finite $\P_{C}$--a.s., hence $\P_{C} (B) = 0$ for every $C \subset_n \Z^d$. However, if $n \geq n_c$, then Theorem~\ref{growth} and the Borel--Cantelli lemma imply that $\P_{C} (B) = 1$ for every $C \subset_n \Z^d$.
\end{proof}

The last of our main results shows that the critical numerosity of CAT coincides with a transition from the recurrence of $\wt{\C}_t$ to its transience. Note that, for CAT to form a set $U$ as $\C_t$ for some $t \geq 1$, it is necessary for $U$ to have a {\em progressive boundary}, i.e., to contain an $m$-element subset $\{x_1,\dots,x_\kk\}$ that satisfies
\begin{equation}\label{seq boundary}
x_j \in \bdy \big( U \setminus \{x_j, \dots, x_\kk\} \big) \quad \text{for $j \in \llbracket 1, m \rrbracket$}.
\end{equation}
In particular, if $\kk=1$, then \eqref{seq boundary} is satisfied by any set $U$ which does not consist exclusively of isolated elements (i.e., elements with no neighbor in $U$). We will prove by induction on $m$ that it is also sufficient, i.e., CAT is irreducible on the collection of such sets.

\begin{definition}[Progressive boundary]
A set $U \in \cal S$ has a progressive boundary if there is a $m$-element subset $\{x_1, \dots, x_m\}$ of $U$ that satisfies \eqref{seq boundary}. We denote by $\Prog_{m,n}$ the collection of $n$-element subsets of $\cal S$ that have progressive boundaries. We denote the collection of equivalence classes of such sets, where each class consists of all translates of an $n$-element set that has a progressive boundary, by $\wtProg_{m,n}$ \eqref{hat equiv class}.
\end{definition}

\begin{theorem}[Transition from positive recurrence to transience]\label{rec to trans}
If $n < 2m+2$, then $\wt{\C}_t$ has a unique stationary distribution, supported on $\wtProg_{m,n}$, to which it converges from any $n$-element subset of $\Z^\dim$. However, if $\nn \geq 2 \kk+2$, then it is transient.
\end{theorem}

Standard Markov chain theory characterizes the existence of a unique stationary distribution in terms of irreducibility, aperiodicity, and positive recurrence. We prove these components in separate propositions.

\begin{proposition}[Irreducibility]\label{irred}
$\wt \C_t$ is irreducible on $\wtProg_{m,n}$ for every $n > m$.
\end{proposition}

We devote Section~\ref{sec irred} to the proof of irreducibility. The key step shows that $\C_t$ can form any configuration from a line, using induction on the number of elements. The fact that $\wt \C_t$ is aperiodic is simpler; we state and prove it here.

\begin{proposition}[Aperiodicity]\label{aper}
$\wt \C_t$ is aperiodic on $\wtProg_{m,n}$ for every $n > m$.
\end{proposition}

\begin{proof}
Suppose that $C = \{je_1: j \in \llbracket 1,n\rrbracket \}$. Note that $C \in \Prog_{m,n}$ because \eqref{seq boundary} is satisfied by setting $y_j = je_1$ for $j \in \llbracket 1, m \rrbracket$. It is aperiodic because $\P_{C} \big( \wt \C_1 = \wt \C_0 \, \big) \geq \P_{C} \big( \AA_0 = \T_0 = y \big) > 0.$ 
Aperiodicity is a class property, so Proposition~\ref{irred} implies that CAT is aperiodic on $\wtProg_{m,n}$.
\end{proof}

\begin{proposition}[Positive recurrence]\label{pos rec}
If $n < 2m+2$, then $\wt \C_t$ is positive recurrent on $\wtProg_{m,n}$.
\end{proposition}

Proposition~\ref{pos rec} follows from Theorem~\ref{drift} because it implies that the diameter of $\wt \C_t$ is a Lyapunov function.

\begin{proposition}[Foster--Lyapunov theorem; Theorem~2.2.4 of \cite{fayollle1995}]\label{fl}
An irreducible Markov chain $(X_t)_{t \geq 0}$ on a countable state space $\cal V$ is positive recurrent if and only if there are 
\[
f : \cal V \to \R_{>0}, \quad g : \cal V \to \Z_{\geq 1}, \quad \eps > 0, \quad \text{and} \quad \text{finite \, $\cal U \subset \cal V$}
\]
such that 
\begin{align}
\forall U \in \cal U, \quad \E_U \big[ f \big( X_{g(U)} \big) \big] &< \infty, \quad \text{and}\label{fl2}\\ 
\forall U \notin \cal U, \quad \E_U \big[ f \big( X_{g(U)} \big) \big] &\leq f(U) - \eps g (U). \label{fl1}
\end{align}
\end{proposition}

\begin{proof}[Proof of Proposition~\ref{pos rec}]
Consider the Markov chain defined by $X_t = \wt \C_t$ on the state space $\cal V = \wtProg_{m,n}$ with $n<2m+2$, and the choices
\[
f (V) = \diam (V), \quad g(V) \equiv n, \quad \eps = 1, \quad \text{and} \quad \cal U = \{V \in \cal V: f(V) \leq b\},
\]
for a yet unspecified $b > 0$. By Proposition~\ref{irred}, $(X_t)_{t \geq 0}$ is irreducible on $\cal V$, which is countable. Hence, by Proposition~\ref{fl}, $\wt \C_t$ is positive recurrent on $\wtProg_{m,n}$ if
\begin{align*}
\forall U: \diam (U) \leq b, \quad \E_U \big[ \diam (\wt \C_n ) \big] &< \infty, \quad \text{and}\\ 
\forall U: \diam (U) > b, \quad \E_U \big[ \diam (\wt \C_n) \big] &\leq \diam (U) - n.
\end{align*}
The first condition holds because $\P_U (\diam (\wt \C_n) \leq \diam (U) + nm) = 1$ by \eqref{amlg prop}. Concerning the second condition, note that Theorem~\ref{drift} applies because $n < 2m+2$ and, together with \eqref{amlg prop}, implies that 
\[
\E_U \big[ \diam (\wt \C_n) \big] \leq (1-p) (\diam (U) + n) + p (2n^2),
\]
for $p > 0$ depending on $n$ only. Consequently, the second condition holds with $b = (2/p-1)n + 2n^2$.
\end{proof}

\begin{proposition}[Transience]\label{trans}
If $n \geq 2m+2$, then $\wt \C_t$ is transient on $\wtProg_{m,n}$.
\end{proposition}

\begin{proof}
By Theorem~\ref{growth} and the Borel--Cantelli lemma, if $C$ has $n \geq 2m+2$ elements, then $\wt \C_t$ makes only finitely many visits to sets in $\wtProg_{m,n}$ with diameter less than $r > 0$, $\P_{C}$--a.s., for every $r$, which implies that $\wt \C_t$ is transient.
\end{proof}

Theorem~\ref{rec to trans} combines Propositions~\ref{irred} through \ref{trans}.

\begin{proof}[Proof of Theorem~\ref{rec to trans}]
By standard Markov chain theory, Propositions~\ref{irred} through \ref{pos rec} imply the existence of a unique stationary distribution for $\wt \C_t$, supported on $\wtProg_{m,n}$, to which it converges from any configuration when $\num < 2m+2$. Propositions~\ref{irred} and \ref{trans} imply that $\wt \C_t$ is transient when $\num \geq 2\kk+2$.
\end{proof}

Theorem~\ref{rec to trans} reflects a high--level strategy for proving the existence of a critical numerosity---namely, to exhibit a Markov chain that transitions from recurrence to transience, as the numerosity of the initial state increases. The Markov chain dichotomy (e.g., \cite[Theorem 1.2.1]{fayollle1995}) explains the virtue of this strategy. 
It states that, if a Markov chain $\X$ is irreducible on a countable set $\cal T$, then either every state in $\cal T$ is recurrent a.s., or every state is transient a.s. This implies that an unbounded function $f: \cal T \to \R_{\geq 0}$ with finite sublevel sets $\{\{T \in \cal T: f(T) \leq i\}\}_{i \in \N}$ satisfies
\begin{equation}\label{dichotomy imp}
\liminf_{t \geq 0} f(\X_t) \stackrel{\text{a.s.}}{=}
\begin{cases}
\inf_{T \in \cal T} f(T) < \infty & \text{$\X$ is recurrent},\\
\infty & \text{$\X$ is transient.}
\end{cases}
\end{equation}
In particular, the Markov chain dichotomy implies that
\begin{equation*}
\forall T \in \cal T, \quad \P_T \left( f(\X_t) \to \infty \right) \in \{0,1\}.
\end{equation*}
Now, suppose that $\X$ has irreducible components $\{\cal T_n \}_{n \geq 2}$, where every element of $\cal T_n$ is an $n$-element set. For simplicity, assume that $\cal T = \cup_{n \geq 2} \cal T_n$. If there is $n_c \in \Z_{\geq 2}$ such that $\X$ is recurrent on $\cal T_n$ when $n < n_c$ and $\X$ is transient on $\cal T_n$ when $n \geq n_c$, then \eqref{dichotomy imp} implies that $\{f(\X_t) \to \infty\}$ has a critical numerosity of $n_c$.

This observation is relevant to Theorem~\ref{phase transition} because $\diam$ is an unbounded function with finite sublevel sets on $\wt{\cal S} = \cup_{n > m} \{\wt S: S \subset_n \Z^d\}$, i.e., configurations viewed up to translation. According to Proposition~\ref{irred}, the irreducible components of $\C$ are $\{\wtProg_{m,n}\}_{n > m}$, where each $\wtProg_{m,n}$ is a subset of $\wt{\cal S}_n = \{\wt S: S \subset_n \Z^d\}$ and the other elements of $\wt{\cal S}_n$ are inessential. By \eqref{dichotomy imp}, if $n > m$, then
\begin{equation*}
\forall \wt S \in \wt{\cal S}_n, \quad \P_{\wt S} \big( \diam (\wt{\C}_t) \to \infty \big) = 
\begin{cases}
0 & \text{$\wt{\C}_t$ is recurrent on $\wtProg_{m,n}$},\\ 
1 & \text{$\wt{\C}_t$ is transient on $\wtProg_{m,n}$}.
\end{cases}
\end{equation*}
The same is true of $\P_S (\diam (\C_t) \to \infty)$ because CAT is invariant under translation and $\diam (\wt S) = \diam (S)$ for every configuration $S$. Consequently, Theorem~\ref{phase transition} follows from Theorem~\ref{rec to trans}.

\subsection{Organization}
Section~\ref{prelims} facilitates a discussion of two key proof ideas in Section~\ref{sec proof ideas}, by gathering preliminary estimates of the CAT transition probabilities. In Section~\ref{sec diameter collapse}, we use the first idea from Section~\ref{sec proof ideas} to quickly prove Theorem~\ref{drift}. Section~\ref{sec diameter growth} begins with a proof of Theorem~\ref{growth}, using the second idea from Section~\ref{sec proof ideas} and assuming a key proposition. The rest of Section~\ref{sec proof ideas} is devoted to a proof of this proposition, which realizes the heuristic that well separated clusters that can persist  move like random walks. Then, Section~\ref{sec cluster formation} proves Theorem~\ref{reach well sep} in two stages, the first of which is another application of the first idea from Section~\ref{sec proof ideas}. The last section, Section~\ref{sec irred}, proves that CAT is irreducible on $\wtProg_{m,n}$, which completes the proof of Theorem~\ref{rec to trans}.

\subsection{Acknowledgements}
I am grateful to Milind Hegde, Joseph Slote, and Bin Yu for encouraging discussions; Sarah Cannon, Joshua Daymude, Brian Daniels, Stephanie Forrest, Dana Randall, Sidney Redner, and Andr{\'e}a Richa for valuable feedback; and Judit Z{\'a}dor and Noah Egan for suggesting the fourth and sixth examples in Table~\ref{exp summary}.

\section{Preliminaries}\label{prelims}

This section collects basic estimates of transition probabilities that facilitate our discussion of the ideas underlying the proofs of Theorems~\ref{drift}--\ref{growth} in Section~\ref{sec proof ideas}. These estimates largely follow from the a.m.l.g.\ of diameter \eqref{amlg prop}, which controls the distance between elements and, through it, the CAT transition probabilities. 

\subsection{A formula for the CAT transition probabilities}\label{tp formula}

Denote by $\AA_{t,j}$ and $\T_{t,j}$ the $j$\textsuperscript{th} elements activated and transported at time $t$. We use $\AA_t$ and $\T_t$ to denote $(\AA_{t,j})_{j \leq m}$ and $(\T_{t,j})_{j \leq m}$ or $\{\AA_{t,j}\}_{j \leq m}$ and $\{\T_{t,j}\}_{j \leq m}$, depending on the context. Additionally, we use $\cal X_t$ to denote the concatenation of $\AA_t$ and $\T_t$, $(\AA_{t,1},\dots,\AA_{t,m},\T_{t,1},\dots,\T_{t,m})$. 
In these terms, 
\begin{equation*}
\C_{t,l} = 
\begin{cases}
\C_t \setminus \{\AA_{t,j}\}_{j\leq l} & l \leq m,\\ 
(\C_t \setminus \AA_t) \cup \{\T_{t,j}\}_{j\leq l-m} & l > m.
\end{cases}
\end{equation*}
In particular, $\C_{t+1} = (\C_t \setminus \AA_t) \cup \T_t$.

We can express the CAT transition probabilities in terms of a family of probability measures on $\Z^d$, indexed by finite, nonempty $U \subset \Z^d$ and $u \in \Z^d$:
\begin{equation}\label{mu def}
\mu_{U,u} (v) = \frac{\phi (u,v)}{\sum_{w \in U} \phi (u,w)}\1_U (v).
\end{equation}
In other words, $\mu_{U,u}$ is the probability measure on $U$ that is proportional to $\phi (u,\cdot)$. 
For ease of notation, we further define $\mu_{U,\infty} (v) = |U|^{-1} \1_U (v)$. In terms of these measures, the components of the transition probability are 
\begin{align*}
\P \left( \AA_t = x \mid \C_t \right) &= \prod_{j=1}^{m} \mu_{ ( \, \C_t \setminus \{x_i\}_{i<j} \, ), \, x_{j-1}} (x_{j}) \quad \text{and}\\ 
\P \left( \T_t = y \mid \C_t, \AA_t\right) &= \prod_{j=1}^m \mu_{\bdy ( \, (\C_t \setminus \AA_t) \, \cup \, \{y_i\}_{i<j} \, ), \, y_{j-1}} (y_j),
\end{align*}
for $x, y \in (\Z^d)^{m}$ and where $x_0 = \infty$ and $y_0 = \AA_{t,m}$. 
The CAT transition probability is
\begin{equation*}
\P (\C_{t+1} = D \mid \C_t = C) = \sum_{(x,y): \, C \to D} \P \left( \AA_t = x, \T_t = y \mid \C_t \right),
\end{equation*}
for configurations $C, D$, where ${C \to D}$ indicates that the sum ranges over $x,y \in (\Z^d)^m$ such that
\[
D = (C \setminus \{x_i\}_{i\leq m}) \cup \{y_i\}_{i\leq m}.
\]

\subsection{Preliminary lemmas}

We now prove three basic estimates of the CAT transition probabilities, using a.m.l.g.\ \eqref{amlg prop} and the definitions of $\phi (u,v)$ \eqref{d def} and $\mu_{U,u} (v)$ \eqref{mu def}.

\begin{lemma}\label{poss implies typ}
Let $C_0, E \subset_n \Z^d$, $x \in (\Z^d)^{2m}$, and $t \geq 1$, and define $\{C_l\}_{1 \leq l \leq 2m}$ according to
\begin{equation}\label{cl def}
C_l = 
\begin{cases}
C_0 \setminus \{x_j\}_{j \leq l} & l \leq m,\\ 
\big( C_0 \setminus \{x_j\}_{j \leq m} \big) \cup \{x_j\}_{j = m+1}^l & l > m.
\end{cases}
\end{equation}
There are $a, b > 0$ such that, if the following probabilities are positive, then they in fact satisfy
\begin{align*}
\P_{C_0} (\C_{0,l+1} = C_{l+1} \mid \C_{0,l} = C_l) &\geq a \|x_{l+1} - x_l\|^{-b},\\ 
\P_{C_0} (\C_1 = C_{2m}) &\geq a \diam (\{x_j\}_{1 \leq j \leq 2m})^{-b}, \quad \text{and}\\ 
\P_{C_0} (\C_t = E) &\geq a^t (t \, \diam (C_0))^{-bt}.
\end{align*}
\end{lemma}

\begin{proof}
We consider each bound in turn. First, by the definition of the CAT dynamics, 
\begin{equation}\label{poss0}
\P_{C_0} (\C_{0,l+1} = C_{l+1} \mid \C_{0,l} = C_l) = 
\begin{cases}
\mu_{C_l,x_l}(x_{l+1}) & l \leq m\\ 
\mu_{\bdy C_l, x_l} (x_{l+1}) & l > m.
\end{cases}
\end{equation}
By \eqref{d def} and \eqref{mu def}, if $U \subset \Z^d$ is finite and nonempty, and if $u, v \in \Z^d$, then either $\mu_{U,u} (v) = 0$ or $\mu_{U,u} (v) \geq |U|^{-1} \| u - v \|^{-\beta}$. Hence, if the conditional probability in \eqref{poss0} is positive, then it satisfies
\begin{equation}\label{poss1}
\P_{C_0} (\C_{0,l+1} = C_{l+1} \mid \C_{0,l} = C_l) \geq (2dn)^{-1} \|x_{l+1} - x_l\|^{-\beta},
\end{equation}
because $|C_l| \leq n$ and $|\bdy C_l| \leq 2dn$.

Second, assume that $\P_{C_0} (\C_1 = C_{2m}) > 0$, in which case there must be $x \in (\Z^d)^{2m}$ and $\{C_l\}_{1 \leq l \leq 2m}$ that satisfy \eqref{cl def} and $\P_{C_0} (\C_{0,1} = C_1,\dots, \C_{0,2m} = C_{2m}) > 0$. We have 
\begin{equation}\label{pre poss2}
\P_{C_0} (\C_1 = C_{2m}) \geq \prod_{1 \leq l \leq 2m-1} \P_{C_0} (\C_{0,l+1} = C_l \mid \C_{0,l} = C_l) \geq (2dn)^{-2m} \prod_{1 \leq l \leq 2m-1} \| x_{l+1} - x_l \|^{-\beta}.
\end{equation}
The Markov property implies the first inequality, as well as the fact that $\P_{C_0} (\C_{0,l+1} = C_{l+1} \mid \C_{0,l} = C_l) > 0$ for each $l \in \llbracket 1, 2m-1 \rrbracket$. Hence, the second inequality follows from \eqref{poss1}. Bounding above $\|x_{l+1} - x_l\|$ by $\diam (\{x_j\}_{1 \leq j \leq 2m})$, we find that
\begin{equation}\label{poss2}
\P_{C_0} (\C_1 = C_{2m}) \geq (2dn)^{-2m} \diam (\{x_j\}_{1 \leq j \leq 2m})^{-2m\beta}.
\end{equation}

Third, assume that $\P_{C_0} (\C_t = C') > 0$, in which case there is a sequence of configurations $D_0, \dots, D_t$ with $D_0 = C_0$ and $D_t = E$, and for which $\P_{D_0} (\C_1 = D_1, \dots, \C_t = D_t) > 0$. Consider $s \in \llbracket 1, t \rrbracket$. Note that, if $y \in (\Z^d)^{2m}$ and $\{D_{s,l}\}_{1 \leq l \leq 2m}$ satisfy \eqref{cl def} in the place of $x$ and $\{C_l\}_{1 \leq l \leq 2m}$, and if $\P_{D_{s-1}} (\C_{0,1} = D_{s,1}, \dots, \C_{0,2m} = D_{s,2m}) > 0$ then the a.m.l.g.\ property \eqref{amlg prop} implies that 
\[
\|y_{l+1} - y_l\| \leq \diam (C_0) + tm \leq 2mt \diam (C_0).
\]
Hence, by \eqref{pre poss2}, 
\[
\P_{D_{s-1}} (\C_1 = D_s) \geq (2dn (2m)^\beta)^{-2m} \big( t \, \diam (C_0) \big)^{-2m\beta}.
\]
By applying the Markov property to each such $s$, we conclude that
\begin{equation}\label{poss3}
\P_{C_0} (\C_t = E) \geq (2dn (2m)^\beta )^{-2m} \big( t \, \diam (C_0) \big)^{-2m\beta t}.
\end{equation}
The three claimed bounds follow from \eqref{poss1}, \eqref{poss2}, and \eqref{poss3}, with $a = (2dn (2m)^\beta)^{-2m}$ and $b = 2m \beta$.
\end{proof}

The probability measure $\mu_{A,u}$ ``picks'' $v \in A$ in proportion to $\phi (u,v)$. If $v$ in fact belongs to $B \subset A$, then $\mu_{B,u} (v) \geq \mu_{A,u} (v)$, since there are fewer elements in $B$. The following result further shows that $\mu_{A,u} (v)$ is at least a fraction of $\mu_{B,u} (v)$, which is closer to one when the diameter of $B \cup \{u\}$ is smaller relative to the distance from $u$ to $A \setminus B$. The proof uses simple bounds on $\phi$.

\begin{lemma}[Ratio estimate]\label{rat est}
Let $B \subseteq A \subset_n \Z^d$, $u \in A^c$, and $v \in B$. Then,
\begin{equation}\label{simple rat bd}
\frac{\mu_{A,u} (v)}{\mu_{B,u} (v)} \geq 1 - n \frac{\diam(B \cup \{u\})^\beta}{\dist(u,A\setminus B)^\beta}.
\end{equation}
\end{lemma}

\begin{proof}
By definition \eqref{mu def},
\[
\mu_{A,u}(v) = \frac{\phi (u,v)}{\sum_{z \in B} \phi (u,z) + \sum_{z \in A \setminus B} \phi (u,z)} \quad \text{and} \quad \mu_{B,u} (v) = \frac{\phi (u,v)}{\sum_{z \in B} \phi (u,z)}.
\]
Since $\frac{1}{1+r} \geq 1 - r$ when $r \geq 0$,
\[
\frac{\mu_{A,u} (v)}{\mu_{B,u} (v)} \geq 1 - \frac{\sum_{z \in A \setminus B} \phi (u,z)}{\sum_{z \in B} \phi (u,z)}.
\]
By assumption, $x \notin A$, so $\phi (u,z) = \| u - z \|^{-\beta}$ for every $z \in A$. Hence, 
\[
\sum_{z \in A\setminus B} \phi(u,z) \leq n \dist (x,A\setminus B)^{-\beta} \quad \text{and} \quad \sum_{z \in B} \phi (u,z) \geq \diam (B \cup \{u\})^{-\beta}.
\]
The claimed bound \eqref{simple rat bd} follows.
\end{proof} 

The probability that activation and transport occur according to a particular sequence $x \in (\Z^d)^{2m}$ is a product of $|C_0|^{-1}$ (to account for the initial activation at $x_1$) and probability measures of the form $\mu_{C_l,x_l} (x_{l+1})$ and $\mu_{\bdy C_l, x_l} (x_{l+1})$ \eqref{poss0}. Lemma~\ref{rat est} suggests that if we add a set $D$ to $C_l$, then the latter components change little, so long as $D$ is relatively far from $C_l$ and $x_{l+1}$ belongs to $C_l$ or $\bdy C_l$. This is the content of the next result; the proof is a simple application of Lemma~\ref{rat est}.

\begin{lemma}\label{lem less simple rat bd}
Let $C_0 \in \cal S$, let $D$ be a finite subset of $\Z^d$ such that $\dist (C_0, D) \geq 2m$, and let $x \in (\Z^d)^{2m}$ satisfy $\P_{C_0} (\cal X_0 = x) > 0$. There is $c>0$ such that 
\begin{equation}\label{less simple rat bd}
\frac{\P_{C_0 \cup D} (\cal X_0 = x)}{\P_{C_0} (\cal X_0 = x)} \geq \frac{|C_0|}{|C_0 \cup D|} \left(1 - c |C_0 \cup D| \frac{\diam (C_0)^\beta}{\dist (C_0,D)^\beta}\right)^{2m}.
\end{equation}
\end{lemma}

\begin{proof}
By the definition of the CAT dynamics,
\begin{equation}\label{rat prod bd}
\frac{\P_{C_0 \cup D} (\cal X_0 = x)}{\P_{C_0} (\cal X_0 = x)} = \frac{|C_0|}{|C_0 \cup D|} \prod_{l=1}^{m-1} \frac{\mu_{C_l \cup D, x_l} (x_{l+1})}{\mu_{C_l,x_l} (x_{l+1})} \prod_{l=m}^{2m-1} \frac{\mu_{\bdy (C_l \cup D), x_l} (x_{l+1})}{\mu_{\bdy C_l, x_l} (x_{l+1})},
\end{equation}
where $(C_l)_{1 \leq l \leq 2m}$ is defined by \eqref{cl def}. 
We aim to apply Lemma~\ref{rat est} with
\[
A = 
\begin{cases}
C_l \cup D & 1 \leq l < m,\\
\bdy (C_l \cup D) & m \leq l < 2m,
\end{cases}
\quad
B = 
\begin{cases}
C_l & 1 \leq l < m,\\
\bdy C_l & m \leq l < 2m,
\end{cases}
\quad u = x_l, \quad \text{and} \quad v = x_{l+1},
\]
for each $l$. The hypotheses of Lemma~\ref{rat est} require that $B \subseteq A$, $u \in A^c$, and $v \in B$. The third requirement is met by \eqref{cl def}. The first two requirements are met because
\[
C_l \subseteq C_l \cup D, \quad \bdy C_l \subseteq \bdy (C_l \cup D), \quad \text{and} \quad x_l \in (C_l \cup D)^c.
\]
Indeed, the second point is due to the assumption that $\dist(C_0,D) \geq 2m$, which implies that $\overline{D}$ is disjoint from $\overline{C}_l$; the third is also due to this assumption, which implies that $x_l \in D^c$, as well as the assumption that $\P_{C_0} (\cal X_0 = x) > 0$, which implies that $x_l \in C_l^c$.

By Lemma~\ref{rat est}, for each $1 \leq l < m$,
\[
\frac{\mu_{C_l \cup D, x_l} (x_{l+1})}{\mu_{C_l,x_l} (x_{l+1})} \geq 1 - |C_l \cup D| \frac{\diam (C_l \cup \{x_l\})^\beta}{\dist (x_l, D)^\beta} \geq 1 - |C_0 \cup D| \frac{\diam (C_0)^\beta}{\dist(C_0,D)^\beta}.
\]
The second inequality holds because $C_l \subseteq C_0$, $x_l \in C_0$, and $C_l \cup \{x_l\} \subseteq C_0$ for each such $l$. Again, by Lemma~\ref{rat est}, for each $m \leq l < 2m$,
\[
\frac{\mu_{\bdy (C_l \cup D), x_l} (x_{l+1})}{\mu_{\bdy C_l, x_l} (x_{l+1})} \geq 1 - |\bdy (C_l \cup D)| \frac{\diam (\bdy C_l \cup \{x_l\})^\beta}{\dist (x_l, D)^\beta} \geq 1 - c |C_0 \cup D| \frac{\diam (C_0)^\beta}{\dist(C_0,D)^\beta}.
\]
The second inequality holds with $c = 2d (2m+6)^\beta$ because, for each such $l$,
\begin{align*}
|\bdy (C_l \cup D) | &\stackrel{\ast}{=} |\bdy C_l| + |D| \leq 2d |C_l| + |D| \leq 2d |C_0| + |D| \stackrel{\ast}{\leq} 2d |C_0 \cup D|,\\
\diam (\bdy C_l \cup \{x_l\}) &\leq \diam (\overline{C}_{2m}) \stackrel{\ast \ast}{\leq} \diam (C_0) + m + 2 \leq (m+3) \diam (C_0),\\ 
\dist (x_l, D) &\stackrel{\ast \ast}{\geq} \dist(C_0, D) - m \stackrel{\ast}{\geq} \dist(C_0, D)/2.
\end{align*}
Note that the equality and inequalities labeled with one asterisk hold by the assumption that $\dist (C_0, D) \geq 2m$ and by a.m.l.g. The inequalities labeled with two asterisks hold by a.m.l.g.\ \eqref{amlg prop}. Substituting these ratio bounds into \eqref{rat prod bd} implies \eqref{less simple rat bd}.
\end{proof}

\section{Proof ideas}\label{sec proof ideas}

\subsection{Idea \#1: A way to obtain diameter--agnostic bounds}

The key feature of Theorems~\ref{drift}~and~\ref{reach well sep} is that their conclusions are agnostic of the diameter of the initial state, despite the dependence of the CAT transition probabilities on the distances between elements. 
The simple idea underlying these results is that it is typical for CAT to form a more ``desirable'' configuration when activation is initiated at an undesirable subset. For example, this is true when it is desirable for a configuration to concentrate its elements in a ``small'' ball or to consist of ``large'' connected components. This idea leads to diameter--agnostic bounds because activation begins uniformly at random. 

\subsubsection{Application to Theorem~\ref{drift}} 
Theorem~\ref{drift} aims to concentrate the elements of a configuration in a ball with a diameter that does not depend on the diameter of the initial state. To obtain a more desirable configuration, we deplete the elements outside of this ball, i.e., through the occurrence of the event
\[
\AddToBall_{x,r} := \big\{ |\C_1 \cap \cal B_x (r+m)| > |\C_0 \cap \cal B_x (r)| \big\},
\] 
which is defined for $x \in \Z^d$ and $r \in \R_{>0}$. The key input to the proof of Theorem~\ref{drift} is a lower bound on the probability that this event occurs.

\begin{lemma}\label{add to ball}
If $n \leq 2m+1$ and if $C \subset_n \Z^d$ satisfies $| C \cap \cal B_x (r) | \in \llbracket m+1, n-1 \rrbracket$ for some $x \in \Z^d$ and $r \geq 1$, then 
\begin{equation}\label{eq add to ball}
\P_{C} (\AddToBall_{x,r}) \geq p,
\end{equation}
for $p>0$ only depending on $r$.
\end{lemma}

\begin{proof}
The expression of {\em Idea \#1} is the claim that
\begin{equation}\label{j1 claim}
\P_C (\AddToBall_{x,r} \mid J > 1) \geq p,
\end{equation}
for some $p>0$ only depending on $r$, where $J := \inf \{j: \AA_{0,j} \in \cal B_x (r)\}$. To prove \eqref{eq add to ball}, it suffices to prove \eqref{j1 claim} instead, since
\[
\P_C (J>1) = \P_C (\AA_{0,1} \notin \cal B_x (r)) \geq \frac1n. 
\]
The inequality follows from the fact that $\AA_{0,1}$ is uniformly random in $C$, at least one element of which lies outside of $\cal B_x (r)$ by assumption. 

We claim that
\begin{equation}\label{a2b incl}
\AddToBall_{x,r} \cap \{J > 1\} \supseteq \AddToBall' \cap \{J > 1\},
\end{equation}
where $\AddToBall'$ is the event
\[
\{J = \infty\} \cup \{1 < J \leq m, \,\, \{\AA_{0,j}\}_{j=J}^m \subset \cal B_x (r), \,\, \T_0 \subseteq \cal B_x (r+m)\}.
\]
In other words, either all activation takes place outside of $\cal B_x (r)$ or, once an element of $\cal B_x (r)$ is activated, all subsequent activation remains within $\cal B_x (r)$ and transport remains within $\cal B_x (r+m)$. We can bound below the probability that $\AddToBall'$ occurs without estimating $\P_C (J=\infty \mid J > 1)$.

Given that $\{J > 1\}$ occurs, $\AddToBall'$ occurs with positive probability because $\cal B_x (r)$ contains more than $m$ elements of $C$ by assumption. Lemma~\ref{poss implies typ} then implies that there are $a,b>0$ such that 
\[
\P_C (\AddToBall' \mid J > 1) \geq a r^{-b}
\]
because, when $\AddToBall'$ occurs, the distance between consecutive elements in $(\AA_{0,J}, \dots, \AA_{0,m}, \T_{0,1}, \dots, \T_{0,m})$ is at most $2(r+m) \lesssim r$. This proves \eqref{j1 claim} with $p = ar^{-b}$, subject to \eqref{a2b incl}.

The inclusion \eqref{a2b incl} holds because, if $\{J=\infty\}$ occurs, then {\em every} element of $C$ outside of $\cal B_x (r)$ is activated; these elements are necessarily transported to $\T_0 \subseteq \cal B_x (r+m)$. Alternatively, if $\AddToBall' \cap \{1 < J \leq m\}$ occurs, then $\AddToBall_{x,r}$ occurs because at least one element outside of $\cal B_x (r)$ is activated and every activated element is transported into $\cal B_x (r+m)$.
\end{proof}

\subsubsection{Application to Theorem~\ref{reach well sep}}

Theorem~\ref{reach well sep} aims to form a configuration that consists of well separated line segments of between $m+1$ and $2m+1$ elements. The first stage of this process forms a configuration consisting of connected components which are ``large,'' in the sense that each has more than $m$ elements. To obtain a more desirable configuration, we deplete the elements of small components, i.e., through the occurrence of the event 
\begin{equation}\label{dsc def}
\DepleteSmallComps_t = \{ |\scr R_{t+1}| < |\scr R_t|\},
\end{equation}
which is defined for $t \in \Z_{\geq 0}$ in terms of $\scr R_t := \{x \in \C_t: |\comp_{\C_t} (x)| \leq m\}$, where $\comp_C (y)$ denotes the set of elements that are connected to $y \in \Z^d$ in $C \subset \Z^d$. The key input to the proof of Theorem~\ref{reach well sep} is a lower bound on the probability that this event occurs. 

\begin{lemma}\label{depl small comps}
There is $p>0$ such that, if $C \subset_n \Z^d$ contains some $x \in \Z^d$ for which $|\comp_C (x)| \leq m$, then 
\begin{equation}\label{eq depl small comps}
\P_C (\DepleteSmallComps_0) \geq p.
\end{equation}
\end{lemma}

\begin{proof}
The expression of {\em Idea \#1} is the claim that
\begin{equation}\label{k1 claim}
\P_C (\DepleteSmallComps_0 \mid K > 1) \geq p,
\end{equation}
for some $p>0$, where $K := \inf \{k: \AA_{0,k} \in \C_0 \setminus \scr R_0\}$. To prove \eqref{eq depl small comps}, it suffices to prove \eqref{k1 claim} instead, since
\[
\P_C (K>1) = \P_C (\AA_{0,1} \in \scr R_0) \geq \frac1n. 
\]
The inequality follows from the fact that $\AA_{0,1}$ is uniformly random in $C$, at least one element of which belongs to $\scr R_0$ by assumption. 

We claim that
\begin{equation}\label{dep incl}
\DepleteSmallComps_0 \cap \{K > 1\} \supseteq \DepleteSmallComps' \cap \{K > 1\},
\end{equation}
where $\DepleteSmallComps'$ is the event
\[
\Big( \{K = \infty\} \cup \big\{1 < K \leq m, \,\, \{ \AA_{0,k} \}_{k=K}^m \cup \{\T_{0,1}\}  \subseteq \comp_{\C_0} (\AA_{0,K})  \Big) \cap \{\T_0 \,\, \text{is connected in $\C_1$} \}.
\]
In other words, the event in parentheses occurs when either activation only occurs at small connected components or, once an element of a large connected component is activated, all subsequent activation takes place at this component, along with the site to which transport first occurs.

Given that $\{K>1\}$ occurs, $\DepleteSmallComps'$ occurs with positive probability because, by assumption, every connected component of $\C_0 \setminus \scr R_0$ has more than $m$ elements. Lemma~\ref{poss implies typ} then implies that there is $p>0$ such that
\[
\P_C (\DepleteSmallComps' \mid K > 1) \geq p
\]
because, when $\DepleteSmallComps'$ occurs, the distance between consecutive elements in the sequence $(\AA_{0,K}, \dots, \AA_{0,m}, \T_{0,1}, \dots, \T_{0,m})$ is at most $n$. This proves \eqref{k1 claim} subject to \eqref{dep incl}.

The inclusion \eqref{dep incl} holds because, if $\{K=\infty, \T_0 \,\,\text{is connected in $\C_1$}\}$ occurs, then {\em every} activated element newly joins or forms a large connected component with the element of $\C_0 \setminus \AA_0$ to whose boundary $\T_{0,1}$ was transported. If $\DepleteSmallComps' \cap \{1 < K \leq m\}$ occurs instead, then $\DepleteSmallComps_0$ occurs because the elements at $\{\AA_{0,k}\}_{k<K}$ newly join the large connected component $\comp_{\C_0} (\AA_{0,K})$, and the rest of the activated elements return to this component.
\end{proof}

\subsection{Idea \#2: An increasingly accurate approximation of the CAT dynamics}

The key idea of the proof of Theorem~\ref{growth} is that, when a configuration of $n \geq 2m+2$ elements consists of clusters of $m+1$ or more elements, the CAT dynamics approximately behaves as if the clusters inhabited separate ``copies'' of $\Z^d$. This approximate dynamics is easier to analyze than the CAT dynamics. Critically, the accuracy of this approximation improves with greater cluster separation. Hence, we can use it to estimate the probability of events that entail sufficiently rapid separation growth, including the event in Theorem~\ref{growth}.

To make this precise, we `lift'' CAT to a Markov chain $\C_t^I = (\C_t^i)_{i \in I}$ on the state space $\Part$, the collection of ordered partitions of configurations in $\cal S$, each part of which has more than $m$ elements. Here, $I$ is an interval of the form $\llbracket 1, k \rrbracket$ for some $k \in \Z_{\geq 1}$. Given $\C_t^I$, to form $\C_{t+1}^I$, we simply apply the CAT dynamics to $\C_t := \cup_{i \in I} \C_t^i \in \cal S$ and update the partition to reflect the new positions of the elements. In other words, we treat the partition like a fixed coloring of the elements, which plays no role in the dynamics. We emphasize that, in the context of a partition $C^I$, we use $C$ to denote $\cup_{i \in I} C^i$.

Next, we define {\em CopyCAT} to be the Markov chain $\D_t^I = (\D_t^i)_{i \in I}$ on the state space of tuples of configurations, $\Tup = \cup_{i \geq 1} {\cal S}^i$, with the following dynamics. Given $\D_t^I$, to form $\D_{t+1}^I$, we apply the CAT dynamics to one entry, $\D_t^{J_t}$, randomly chosen with a probability proportional to $|\D_t^i|$. In other words, 
\begin{equation}\label{copycat def}
\D_{t+1}^i := 
\begin{cases}
\mathrm{CAT} (\D_t^i) & i = J_t,\\
\D_t^i & i \neq J_t,
\end{cases}
\end{equation}
where, for any $C,C' \in \cal S$ and $i \in I$, 
\[
\P (\mathrm{CAT} (C) = C') = \P_C (\C_1 = C') \quad \text{and} \quad \P (J_t = i \mid \D_t) = \frac{|\D_t^i|}{\sum_{i \in I} |\D_t^i|}.
\]
We use $\cal Y_t$ to denote the sequence of elements of $\D_t^I$ that are activated and transported elements at time $t$, i.e., the CopyCAT analogue of $\cal X_t$.

\begin{figure}
\includegraphics[width=0.65\linewidth]{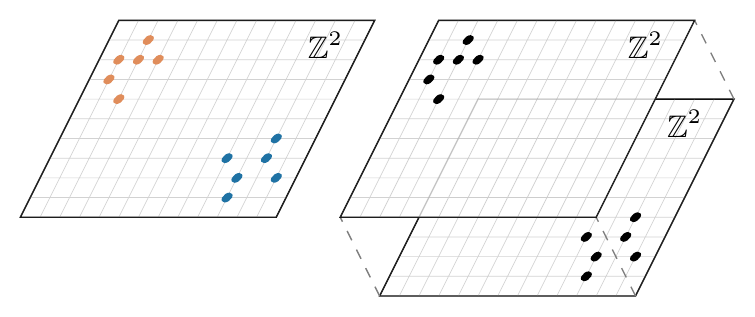}
\caption{The lifting of CAT corresponds to a coloring which is ignored by the dynamics (left), while a CopyCAT configuration may have elements in different ``copies'' of $\Z^2$ (right).}
\label{multicat fig}
\end{figure}

To state the key approximation result, we define the separation of $U^I \in \Tup$ to be
\[
\sep (U^I) := \min_{i \in I} \dist (U^i, U^{\neq i}),
\]
where $U^{\neq i} = \cup_{j \in I\setminus\{i\}} U^j$. 
For $a, b \in \R_{>0}$, we define tuples with absolute separation $a$ and {\em relative} (to diameter) {\em separation} $b$ by 
\begin{equation}\label{tup def}
\Tup_{a,b} := \left\{U^I \in \Tup: \sep (U^I) \geq a \quad \text{and} \quad \forall i \in I, \,\, \diam (U^i) \leq b \log \dist (U^i,U^{\neq i})\right\}.
\end{equation}
Note that $\Tup_{a,b} \subseteq \Part$ since $a$ is positive. For $\gamma \in \R_{>0}$, $\delta \in (0,\frac12)$, and $t \in \Z_{\geq 0}$, we define $\mathrm{Growth}_{a, b, \gamma, \delta; t}$ to be the set of sequences $(U_s^I)_{s \geq 0} \in \Tup_{a,b}^\N$ such that
\[
\forall s \leq t, \quad \sep (U_s^I) \geq \gamma s^{1/2 - \delta}.
\]
Additionally, we define $\mathrm{Growth}_{a, b, \gamma, \delta} = \cap_{t \geq 0} \mathrm{Growth}_{a, b, \gamma, \delta; t}$.

\begin{proposition}[Key approximation result]\label{copycat growth imp}
Let $b, \gamma > 0$ and $\delta \in (0,\frac12 - \frac{1}{\beta})$. If $a\geq 2m$ is sufficiently large, then 
\begin{equation*}
\P_{C^I} \left( (\C_t^I)_{t \geq 0} \in \mathrm{Growth}_{a,b,\gamma,\delta} \right) \geq (1-o_a (1)) \, \P_{C^I} \left( (\D_t^I)_{t \geq 0} \in \mathrm{Growth}_{a,b,\gamma,\delta} \right)
\end{equation*} 
for every $C^I \in \Tup_{a,b}$.
\end{proposition}

The main input to the proof of Proposition~\ref{copycat growth imp} is a comparison of the one-step transition probabilities of CAT and CopyCAT. 

\begin{proposition}[One-step approximation]\label{one step approx}
Let $b>0$ and let $U^I, V^I \in \Tup_{2m,b}$ satisfy $\P_{U^I} (\C_1^I = V^I) > 0$. There is $c>0$ such that 
\begin{equation}\label{eq one step approx}
\P_{U^I} \left( \C_1^I = V^I \right) \geq \left( 1 - c |U| \, \sep (U^I)^{-\beta} \big( \log \sep (U^I) \big)^\beta  \right)^{2m} \P_{U^I} \left( \D_1^I = V^I \right).
\end{equation}
\end{proposition}

Proposition~\ref{copycat growth imp} is a simple consequence of the Markov property, Proposition~\ref{one step approx}, and the definition of $\mathrm{Growth}_{a, b, \gamma, \delta}$.

\begin{proof}[Proof of Proposition~\ref{copycat growth imp}]
Let $t \geq 0$. By the Markov property and Proposition~\ref{one step approx}, we have
\begin{align*}
\P_{C^I} \left( (\C_s^I)_{s \geq 0} \in \mathrm{Growth}_{a, b, \gamma, \delta; t} \right)
= &\sum \P_{C^I} \left( \C_s^I = D_s^I, s \geq 0 \right)\\ 
\geq w_t & \sum \P_{C^I} \left( \D_s^I = D_s^I, s \geq 0 \right) = w_t \P_{C^I} \left( (\D_s^I)_{s \geq 0} \in \mathrm{Growth}_{a, b, \gamma, \delta; t} \right),
\end{align*}
where the sums range over $(D_s^I)_{s \geq 0} \in \mathrm{Growth}_{a, b, \gamma, \delta; t}$. Here, $w_t = \prod_{s \leq t} \left( 1 - c r_a (s)^{-\beta} (\log r_a (s))^\beta \right)^{2m}$, where $c>0$ is the number of the same name in \eqref{eq one step approx} and $r_a (s) = \max\{a,\gamma s^{1/2-\delta}\}$. Since $\mathrm{Growth}_{a, b, \gamma, \delta; t}$ decreases with $t$, the continuity of $\P_{C^I}$ implies that
\[
\P_{C^I} \left( (\C_s^I)_{s \geq 0} \in \mathrm{Growth}_{a, b, \gamma, \delta} \right) \geq \lim_{t\to\infty} w_t \, \P_{C^I} \left( (\D_s^I)_{s \geq 0} \in \mathrm{Growth}_{a, b, \gamma, \delta} \right).
\]
It is easy to see that, since $\beta > 2$, if $\delta < \frac12 - \frac{1}{\beta}$, then $\sum_{s \geq 0} r_a(s)^{-\beta} (\log r_a (s))^\beta = o_a (1)$, which implies that $\lim_{t\to\infty} w_t = (1 - o_a (1))$.
\end{proof}

Proposition~\ref{one step approx} is an application of Lemma~\ref{lem less simple rat bd}.

\begin{proof}[Proof of Proposition~\ref{one step approx}]
Fix $U^I, V^I \in \Tup_{2m,b}$ and define $\cal G$ and $\cal H$ to be the activation and transport sequences $x \in (\Z^d)^{2m}$ that form $V^I$ from $U^I$ under the CAT and CopyCAT dynamics, respectively:
\begin{align*}
\cal G &:= \left\{x \in (\Z^d)^{2m}: \{\C_0^I = U^I, \cal X_0 = x\} \subseteq \{ \C_1^I = V^I \} \,\, \text{and} \,\, \P_{U^I} (\cal X_0 = x) > 0 \right\},\\ 
\cal H &:= \left\{x \in (\Z^d)^{2m}: \{\D_0^I = U^I, \cal Y_0 = x\} \subseteq \{ \D_1^I = V^I \} \,\, \text{and} \,\, \P_{U^I} (\cal Y_0 = x) > 0 \right\}.
\end{align*}
Observe that $\cal G \supseteq \cal H$ because $V^I$ is separated by more than $m$. Indeed, so long as the clusters are too far apart to intersect after one step of the CopyCAT dynamics, the only difference between $\cal G$ and $\cal H$ is that there are $x \in \cal G \setminus \cal H$ that correspond to activation at more than one cluster.

We calculate
\begin{equation}\label{one step seq1}
\P_{U^I} (\C_1^I = V^I) 
= \sum_{x \in \cal G} \P_{U^I} (\cal X_0 = x) 
\geq \sum_{x \in \cal H} \P_{U^I} (\cal X_0 = x) 
\geq \sum_{x \in \cal H} w_x \, \frac{|U^{i_x}|}{|U|} \, \P_{U^{i_x}} (\cal X_0 = x),
\end{equation}
where $i_x$ is the cluster of $U^I$ to which $x_1$ belongs and 
\[
w_x = \left( 1 - c |U| \frac{\diam (U^{i_x})^\beta}{\dist (U^{i_x}, U^{\neq i_x})^\beta} \right)^{2m},
\]
in terms of the quantity $c$ from Lemma~\ref{lem less simple rat bd}. 
The equality in \eqref{one step seq1} holds by definition, the first inequality by $\cal G \supseteq \cal H$, and the second inequality by Lemma~\ref{lem less simple rat bd} applied with $C_0 = U^{i_x}$ and $D = U^{\neq i_x}$. The use of this lemma requires that $\dist (C_0, D) \geq 2m$ and $\P_{C_0} (\cal X_0 = x) > 0$. The former requirement is satisfied because $U^I$ has separation of at least $2m$, while the latter is satisfied because $\P_{U^I} ( \cal Y_0 = x)$ is positive by virtue of $x \in \cal H$, which implies the same is true of $\P_{C_0} (\cal X_0 = x)$ because 
\[
\P_{U^I} (\cal Y_0 = x) = \frac{|U^{i_x}|}{|U|}  \, \P_{U^{i_x}} (\cal X_0 = x)
\]
by \eqref{copycat def}. By definition, $\P_{U^I} (\D_1^I = V^I) = \sum_{x \in \cal H} \P_{U^I} (\cal Y_0 = x)$, hence by substituting the preceding display into \eqref{one step seq1}, we conclude that
\[
\P_{U^I} (\C_1^I = V^I) \geq \min_{x \in \cal H} w_x \, \P_{U^I} (\D_1^I = V^I).
\]
The claimed bound \eqref{eq one step approx} then follows from 
\[
\min_{x \in \cal H} w_x \geq \min_{j \in I} \left( 1 - c |U| \frac{\big( b \log \dist (U^j, U^{\neq j}) \big)^\beta}{\dist (U^j, U^{\neq j})^\beta} \right)^{2m} \geq \left( 1 - cb^\beta |U| \frac{\big( \log \sep (U^I) \big)^\beta}{\sep (U^I)^\beta} \right)^{2m}.
\]
The first inequality holds because $U^I \in \Tup_{\cdot,b}$, while the second holds because $\dist (U^j,U^{\neq j}) \geq \sep (U^I)$ and $\frac{\log x}{x}$ is decreasing in $x > e$.
\end{proof}

\section{Diameter collapse}\label{sec diameter collapse}

Theorem~\ref{drift} states that, if $n \leq 2m+1$, then the diameter of $\C_n$ is at most $2n^2$ with a probability of at least some $p>0$. 
We prove Theorem~\ref{drift} with an estimate of $\tau_x$, defined for $x \in \Z^d$ as the first time $t$ that the number $N_x (t)$ of elements of $\C_t$ in the ball $\cal B_x (tm+n)$ equals $n$:
\begin{equation}\label{tx and nx}
\tau_x := \inf \{t \geq 0: N_x (t) = n\} \quad \text{where} \quad N_x (t) := | \C_t \cap \cal B_x (tm+n) |.
\end{equation}
When $\{\tau_x \leq m\}$ occurs, the diameter of $\C_n$ satisfies
\begin{equation}\label{tx bd}
\diam (\C_n) \leq 2 (\tau_x m + n) + (n - \tau_x) m = \tau_x m + n(m+2) \leq 2n^2.
\end{equation}
The first inequality follows from the definition of $\tau_x$ and a.m.l.g.\ \eqref{amlg prop}. The second follows from bounding $\tau_x$ by $m$, $m$ by $n-1$, and $2n^2 - n+1$ by $n^2$.

\subsection{Application of Lemma~\ref{add to ball}}

The next proposition states that $\{\tau_x \leq m\}$ typically occurs under $\P_C$ for certain $C \subset_n \Z^d$. The proof is an application of Lemma~\ref{add to ball}. In the context of \eqref{tx bd}, this proves Theorem~\ref{drift} for such $C$.

\begin{proposition}\label{tau m}
Let $C \subset_n \Z^d$ satisfy $|C \cap \cal B_x (n)| \geq m+1$ for some $x \in \Z^d$. If $n \in \llb m+1,2m+1 \rrb$, then there is $p>0$ such that
\begin{equation}\label{eq tau m}
\P_{C} ( \tau_x \leq m) \geq p.
\end{equation}
\end{proposition}

\begin{proof}
For $0 \leq j < k \leq m$, denote by $A_{j,k}$ the event that $N_x (t)$ increases with each step from $t = j$ to $t = k$, i.e., $A_{j,k} = \{N_x (j) < \cdots < N_x (k) \}$. For notational convenience, define $A_{0,0}$ to be the sample space. The key observation is that 
\begin{equation}\label{a0tau}
\P_{C} (\tau_x \leq m \mid \tau_x > 0) \geq \P_{C} (A_{0,\tau_x} \mid \tau_x > 0).
\end{equation}
Indeed, $\cal B_x (n)$ contains at least $m+1$ elements of the $n \leq 2m+1$ elements of $C$ by assumption, so $N_x (t)$ cannot increase more than $n - (m+1) \leq m$ times consecutively.

Lemma~\ref{add to ball} implies that the right hand side of \eqref{a0tau} is at least $q^m$ for some $q > 0$, hence \eqref{eq tau m} holds with $p = q^m$. 
To see why, define $(f_j,g_j,r_j)_{1 \leq j \leq m}$ according to
\begin{align*}
f_j &:= \P_{C} (A_{j-1,\tau_x} \mid \tau_x > j-1, A_{0,j-1}),\\ 
g_j &:= \P_{C} (A_{j-1,j} \mid \tau_x > j-1, A_{0,j-1}),\\ 
r_j &:= \P_{C} (\tau_x > j \mid \tau_x > j-1, A_{0,j-1}),
\end{align*}
and note that $f_1$ is the right hand side of \eqref{a0tau}. It is easy to verify that these quantities solve the following system of equations:
\begin{equation*}
f_m = g_m \quad \text{and} \quad \forall 1\leq j \leq m-1, \quad f_j = g_j (1 - r_j + r_j f_{j+1}).
\end{equation*}
In particular, $f_j \geq g_j f_{j+1}$ for each $j \leq m-1$, regardless of the value of $r_j$. To prove that $f_1 \geq q^m$, it therefore suffices to show that $g_j \geq q$ for each $j \in \llb 1, m \rrb$. Given $\C_j$, when $\{\tau_x > j-1, A_{0,j-1}\}$ occurs, $N_x (j)$ belongs to $\llb m+1,n-1\rrb$, hence $\C_j$ satisfies the hypotheses on $C$ in the statement of Lemma~\ref{add to ball}. Consequently, by applying the Markov property at $t = j$, Lemma~\ref{add to ball} implies that $g_j \geq q$ for some $q > 0$.
\end{proof}

\subsection{Proof of Theorem~\ref{drift}}
We aim to apply Proposition~\ref{tau m}, but its hypotheses stipulate that $C$ must permit an $x \in \Z^d$ for which $|C \cap \cal B_x (n)| \geq m+1$. To satisfy this hypothesis, we apply the Markov property to $t=1$ and work with $\C_1$ in the place of $C$, given that $G = \{\T_0 \subset \cal B_{\T_{0,1}} (n)\}$ occurs. (When $G$ occurs, $\C_1$ satsifies $|\C_1 \cap \cal B_{\T_{0,1}} (n)| \geq m+1$.) Lemma~\ref{poss implies typ} implies that there is $q_1 > 0$ such that $\P_{C} (G) \geq q_1$, hence it suffices to prove that
\begin{equation}\label{c1 instead of c0}
\P_{\C_1} (\diam (\C_{n-1}) \leq 2n^2) \geq q_2 \1_G,
\end{equation}
for some $q_2 > 0$, as then \eqref{eq drift} will hold with $p = q_1 q_2$.

Given $\C_1$, when $G$ occurs, let $x \in \Z^d$ satisfy $|\C_1 \cap \cal B_x (n)| \geq m+1$ and define $N_x (t)$ and $\tau_x$ as in \eqref{tx and nx}. 
By Proposition~\ref{tau m}, there is $q_2 > 0$ such that 
\begin{equation}\label{tau m bd}
\P_{\C_1} (\tau_x \leq m) \geq q_2 \1_G.
\end{equation}
This bound implies \eqref{c1 instead of c0} because, when $\{\tau \leq m\}$ occurs, we have 
\[
\diam (\C_{n-1}) \leq 2(\tau_x m + n) + (n - 1 - \tau_x)m \leq 2n^2.
\]
As in \eqref{tx bd}, the first inequality holds by the definition of $\tau_x$ and \eqref{amlg prop}, while the second holds because $\tau_x \leq m \leq n-2$ and $2n^2 - 5n + 6 \leq 2n^2$.

\section{Diameter growth}\label{sec diameter growth}

\subsection{Proof of Theorem~\ref{growth}}
Recall that for $a,b,\gamma > 0$ and $\delta \in (0,\frac12)$ the set $\Growth_{a,b,\gamma,\delta}$ consists of $(C_t^I)_{t \geq 0} \in \Tup_{a,b}^\N$ such that
\[
\forall t \geq 0, \quad \sep (C_t^I) \geq \gamma t^{\frac12 - \delta}.
\]
This section proves that it is typical for CopyCAT to exhibit this kind of separation growth, so long as the initial partition is sufficiently separated.

\begin{proposition}\label{d growth}
Let $C \subset_n \Z^d$ for $n \geq 2m+2$. There are $b,\gamma > 0$ such that, if $C \in \Stable_a$ for sufficiently large $a \in \Z_{\geq 0}$, then for every $\delta \in (0,\frac12)$, 
\[
\P_{C^I} \left( (\D_t^I)_{t \geq 0} \in \Growth_{\tilde a,b,\gamma,\delta}\right) \geq \frac14,
\]
where $\tilde a = a^{0.99}$ and $C^I$ is any partition of $C$ into $a$-separated line segments that satisfy \eqref{reach well sep conds}.\footnote{Note that $C$ must have such a partition because $C \in \Stable_a$.}
\end{proposition}

We use $a^{0.99}$ in the place of $a$ because it is typical for the separation to temporarily decrease below $a$ by a constant factor of $a$, which we can absorb as $a^{0.01}$ for sufficiently large $a$.

Theorem~\ref{growth} follows from Theorem~\ref{reach well sep}, the key approximation result (Proposition~\ref{copycat growth imp}), and Proposition~\ref{d growth}.

\begin{proof}[Proof of Theorem~\ref{growth}]
Let $C \subset_n \Z^d$ for $n \geq 2m+2$, let $b$ and $\gamma$ be the numbers of the same name in Proposition~\ref{d growth}, let $\delta \in (0,\frac12 - \frac{1}{\beta})$, and let $a \in \Z_{\geq 2m}$, $\tilde a = a^{0.99}$, and $s = 2n^2 a$. By the Markov property applied to time $s$,
\[
\P_{C} \left( \diam (\C_t) \geq \gamma (t- s)^{\frac12 - \delta} \,\,\text{for every $t \geq s$} \right)  
\geq \E_C \left[ \P_{\C_s^I} \left( (\C_t^I)_{t \geq 0} \in \Growth_{\tilde a,b,\gamma,\delta} \right) \1_{\{\C_s \in \Stable_{a}\}} \right],
\]
where, given $\C_s \in \Stable_{a}$, $\C_s^I$ is any partition of $\C_s$ that satisfies \eqref{reach well sep conds}. By Propositions \ref{copycat growth imp} and \ref{d growth}, if $a$ is sufficiently large, then 
\begin{equation}\label{eigth bd}
\P_{D^I} \left( (\C_t^I)_{t \geq 0} \in \mathrm{Growth}_{\tilde a,b,\gamma,\delta} \right) \geq \frac12 \, \P_{D^I} \left( (\D_t^I)_{t \geq 0} \in \mathrm{Growth}_{\tilde a,b,\gamma,\delta} \right) \geq \frac18,
\end{equation}
for every partition $D^I$ of $D \in \Stable_{a}$ that satisfies \eqref{reach well sep conds}. Note that we need $D^I \in \Tup_{\tilde a,b}$ to justify the first inequality with Proposition~\ref{copycat growth imp}. This holds when $\tilde a > e^{n/b}$ because $D^I \in \Tup_{\tilde a, \tilde b}$ for $\tilde b = \frac{n}{\log \tilde a}$ by \eqref{reach well sep conds} and $\Tup_{\tilde a, \tilde b} \subseteq \Tup_{\tilde a,b}$ by \eqref{tup def}.

By the preceding bounds and Theorem~\ref{reach well sep}, there is $p>0$ such that
\[
\P_{C} \left( \diam (\C_t) \geq \gamma (t- s)^{\frac12 - \delta}, \, \, t \geq s \right) \geq \frac18 \, \P_C (\C_s \in \Stable_a) \geq \frac18 \, p^a.
\]
Although we assumed $\delta < \frac12 - \frac{1}{\beta}$ to apply Proposition~\ref{copycat growth imp}, this bound implies that \eqref{eq growth} holds for every $\delta \in (0,\frac12)$.
\end{proof}

In the rest of this section, we will prove Proposition~\ref{d growth} by analyzing random walks that will arise as the differences of elements that are representative of two clusters, evolving according to the CopyCAT dynamics and viewed at a sequence of renewal times. Standard estimates for random walk will imply that if these representatives are sufficiently separated, then their separation will grow as $\Omega(t^{1/2-\delta})$ with high probability in the initial separation. 

\subsection{Random walks associated to pairs of clusters}

Given $\D_t^I$, for each $i \in I$ and $t \geq 0$, we will represent $\D_t^i$ by $M_t^i$, defined as the element of $\D_t^i$ that is least in the lexicographic order on $\Z^d$.\footnote{We would not benefit from representing clusters by, e.g., their centers of mass, as the clusters will have diameters that are small relative to their separation. For this reason, elements will be equally representative of the clusters to which they belong and, unlike centers of mass, they will necessarily belong to $\Z^d$, which will be convenient.} We will view each $M_t^i$ at the consecutive times at which all clusters form line segments parallel to $e_1$, i.e., sets in $\cal L$. Define this sequence of times by
\begin{align*}
\xi_0 &:= \inf \big\{ t \geq 0: \D_t^i \in \cal L \,\,\text{for every $i \in I$} \big\}, \quad \text{and}\\ 
\forall l \geq 1, \quad \xi_l &:= \inf \big\{ t > \xi_{l-1}: \D_t^i \in \cal L  \,\,\text{for every $i \in I$}\big\}.
\end{align*}

We will study each pair of distinct clusters $i,j \in I$ separately, through the random walk $(S_k^{ij})_{k \geq 0}$, defined according to
\begin{equation}\label{rw def}
S_k^{ij} := \left( M_{\xi_0}^i - M_{\xi_0}^j \right) + \sum_{l=1}^k \left( M_{\xi_l}^i - M_{\xi_l}^j  - M_{\xi_{l-1}}^i + M_{\xi_{l-1}}^j\right). 
\end{equation}
In part to ensure that $\xi_0 \equiv 0$, we will study $(S_k^{ij})_{k \geq 0}$ under $\P_{D^I}$ for tuples $D^I \in \Tup$ that satisfy
\begin{equation}\label{stable def}
|I| \geq 2 \quad \text{and} \quad \forall i \in I, \quad D^i \in \cal L \quad \text{and} \quad |D^i| \in \llbracket m+1, 2m+1 \rrbracket.
\end{equation}

\begin{proposition}\label{wm is rw}
If $D^I\in \Tup$ satisfies \eqref{stable def}, then, for every distinct $i, j \in I$, the distribution of $(S_k^{ij})_{k \geq 0}$ under $\P_{D^I}$ is that of a symmetric, aperiodic, and irreducible random walk on $\Z^d$.
\end{proposition}

\begin{proof}
First, $S_k^{ij}$ is a symmetric random walk because the CopyCAT transition probability is translation invariant. Second, it is aperiodic because for each $t \geq 0$ every cluster of $\D_t^I$ belongs to $\Prog_m$, $\P_{D^I}$--a.s. Hence, by the progressive boundary property \eqref{seq boundary}, we have $\P_{D^I} (\D_1 = \D_0) > 0$, which implies $\P_{D^I} (S_1^{ij} = S_0^{ij}) > 0$. Third, concerning irreducibility, it suffices to show that
\[
\forall l \in \llbracket 1, d \rrbracket, \quad \P_{D^I} \left( \xi_1 = 5, \D_5^i = \D_0^i + e_l, \D_5^j = \D_0^j \right) > 0,
\]
since the occurrence of the event in question implies that $S_1^{ij} = S_0^{ij} + e_l$. We can realize this event in the following way. 
Let $l \in \llbracket 1,d \rrbracket$ and suppose w.l.o.g.\ that $\D_0^i = L_{1,n_i}$, where $n_i = |\D_0^i|$. In the first step, we require that $\D_1^j \notin \cal L$, so that we do not accidentally reach $\xi_1$ before $t = 5$. Then, in the next two steps, we form $\D_1^i = \D_0^i$ into $\{e_1\} \cup (L_{1,n_i - 1} + e_l)$, which is possible because $n_i \leq 2m+1$ by assumption. We reach $\D_4^i = L_{1,n_i} + e_l = \D_0^i + e_l$ in the fourth step and, in the fifth step, we reverse the first. This sequence of steps results in $\xi_1 = 5$ and $S_1^{ij} = S_0^{ij} + e_l$. 
\end{proof}

The next result shows that the norm of the random walk associated to a pair of clusters is a proxy for their separation, as long as the time between consecutive steps is relatively small. We analyze the time between consecutive renewals in the next section. 

\begin{proposition}\label{lem dist bd}
Let $l \geq 1$, let $t \in \llbracket \xi_{l-1}, \xi_l \rrbracket$, and let $D^I \in \Tup$ satisfy \eqref{stable def}. Then, 
\begin{equation}\label{dist bd}
\dist (\D_t^i, \D_t^j) \geq \big\| S_{l-1}^{ij} \big\| - O \left( (\xi_l - \xi_{l-1})^2 \right)
\end{equation}
for every pair of distinct $i,j \in I$, $\P_{D^I}$--a.s.
\end{proposition}

\begin{proof}
We will show that
\begin{equation}\label{m diff bd}
\big\| (M_t^i - M_t^j) - (M_{\xi_{l-1}}^i - M_{\xi_{l-1}}^j) \big\| \lesssim (\xi_l - \xi_{l-1})^2.
\end{equation}
Noting that $\|S_{l-1}^{ij} \| = \big\| M_{\xi_{l-1}}^i - M_{\xi_{l-1}}^j \big\|$, the claimed bound \eqref{dist bd} then follows from the triangle inequality, \eqref{amlg prop}, and \eqref{m diff bd}:
\begin{align*}
\dist (\D_t^i, \D_t^j) &\geq \| M_t^i - M_t^j\| - \diam (\D_t^i) - \diam (\D_t^j)\\ 
&\geq \| M_t^i - M_t^j\| - 2(\xi_l - \xi_{l-1}) - |D^i| - |D^j|\\ 
&\geq \big\| M_{\xi_{l-1}}^i - M_{\xi_{l-1}}^j \big\| - O \left( (\xi_l - \xi_{l-1})^2 \right).
\end{align*}

To show \eqref{m diff bd}, we begin by writing the difference as a telescoping sum and applying the triangle inequality
\[
\big\| (M_t^i - M_t^j) - (M_{\xi_{l-1}}^i - M_{\xi_{l-1}}^j) \big\| \leq \sum_{s=\xi_{l-1}}^{t-1} \left( \big\| M_{s+1}^i - M_s^i \big\| + \big\| M_{s+1}^j - M_s^j \big\| \right).
\]
It is easy to see that $\| M_{s+1}^i - M_s^i\|$ is at most $\diam (\D_s^i) + m$, so \eqref{amlg prop} implies that
\[
\| M_{s+1}^i - M_s^i\| \leq (\xi_l-\xi_{l-1}) + \diam (\D_{\xi_{l-1}}^i) + m,
\]
for $s \in \llbracket \xi_{l-1},\xi_l \rrbracket$. 
Since $\D_{\xi_{l-1}}^i$ is a line segment of length $|D^i|$, it has a diameter of $|D^i|$ and, as $\xi_l - \xi_{l-1} \geq 1$, we can simply bound $\|M_{s+1}^i - M_s^i\|$ by $(|D^i|+m+1) (\xi_l - \xi_{l-1})$. Applying this to the preceding sum, we conclude
\[
\big\| (M_t^i - M_t^j) - (M_{\xi_{l-1}}^i - M_{\xi_{l-1}}^j) \big\| \leq (|D^i| + |D^j| + 2m + 2) (\xi_l - \xi_{l-1})^2.
\]
\end{proof}

\subsection{Renewal time estimates}

The main result of this subsection states that it is typical for $\xi_l - \xi_{l-1}$ to grow like $\log l$, in the following sense.

\begin{proposition}\label{prop short rets}
There are $\alpha_1, \alpha_2 \in \R_{>0}$ such that, if $D^I \in \Tup$ satisfies \eqref{stable def}, then
\begin{equation}\label{short rets}
\P_{D^I} \left( \cap_{l=1}^\infty F_l \right) \geq \frac12,
\end{equation}
where $F_l := \{\xi_l - \xi_{l-1} \leq \alpha_1 \log (\alpha_2 l)\}$.
\end{proposition}

The proof is a consequence of the fact that $\xi_1 - \xi_0$ is exponentially tight.

\begin{proposition}\label{exp tail}
There are $c_1,c_2 \in \R_{>0}$ such that, if $D^I \in \Tup$ satisfies \eqref{stable def}, then
\begin{equation}\label{return time bd}
\forall t \geq 0, \quad \P_{D^I} ( \xi_1 > t) \leq c_1 e^{-c_2t}.
\end{equation}
\end{proposition}

\begin{proof}[Proof of Proposition~\ref{prop short rets}]
By the strong Markov property applied to $\xi_{l-1}$ and \eqref{return time bd},
\[
\P_{D^I} (F_l^c) \leq c_1 e^{-c_2 \alpha_1 \log (\alpha_2 l)}.
\]
Taking $\alpha_1 = \frac{2}{c_2}$ and $\alpha_2 = \frac{2}{\sqrt{c_1}}$ gives a bound of $(2l)^{-2}$. A union bound then gives \eqref{short rets}:
\[
\P_{D^I} \left( \cap_{l=1}^\infty F_l \right) \geq 1 - \sum_{l=1}^\infty \P_{D^I} (F_l^c) \geq 1 - \sum_{l=1}^\infty (2l)^{-2} \geq \frac12.
\]
\end{proof}

The time between consecutive renewals is exponentially tight for two reasons. First, let $D^I \in \Tup$ with $\sum_{i \in I} |D^i| = n$. By the definition of CopyCAT \eqref{copycat def}, 
\begin{equation}\label{part of d is cat}
\P_{D^I} \left( (\D_s^i)_{s \leq t} = (C_s)_{s \leq t} \right) \geq n^{-t} \, \P_{D^i} \left( (\C_s)_{s \leq t} = (C_s)_{s \leq t} \right),
\end{equation}
for every $i \in I$, $t \geq 1$, and $(C_s)_{s \leq t} \in {\cal S}^t$. 
Second, CAT returns to $\cal L$ after a fixed number of steps with a probability of at least some $q>0$ under $\P_C$, uniformly for $C \subset_n \Z^d$ with $n \in \llb m+1,2m+1\rrb$ elements.

\begin{proposition}\label{cor form line}
There is a positive number $p>0$ such that, if $C \subset_n \Z^d$ has $n \in \llb m+1,2m+1\rrb$ elements, then
\begin{equation}\label{form l bd}
\P_{C} ( \C_{n+2} \in \cal L) \geq p.
\end{equation}
\end{proposition}

The time $n+2$ in \eqref{form l bd} comes from waiting $n$ steps to have a small (i.e., bounded by a function of $n$) diameter (Theorem~\ref{drift}), after which we can dictate two transport steps to form a line, without introducing a dependence on $\diam (C)$ into $p$. (We shortly explain how to form a line in two steps.)

\begin{proof}[Proof of Proposition~\ref{exp tail}]
Let $E^I \in \Tup$ satisfy $|E^i| \in \llbracket m+1,2m+1 \rrbracket$ for every $i \in I$ and denote $s = \sum_{i \in I} (|E^i| + 2)$. By \eqref{part of d is cat} and Proposition~\ref{cor form line}, there is $p > 0$ such that
\[
\P_{E^I} \left( \D_s^i \in \cal L \,\,\text{for all $i \in I$} \right) \geq p,
\]
for any such $E^I$. 
Note that $s \leq 2n$ because $E^I$ has at most $[\frac{n}{m+1}] \leq \frac{n}{2}$ clusters and $\sum_{i \in I} |E^i| = n$. 
Consequently, if $D^I \in \Tup$ satisfies \eqref{stable def}, then
\[
\forall t \geq 0, \quad \P_{D^I} (\xi_1 > t) \leq (1-p)^{[\frac{t}{s}]} \leq (1-p)^{\frac{t}{2n} - 1}.
\]
The first inequality holds by applying the preceding bound and the Markov property every $s$ steps. The second inequality uses the fact that $[r] \geq r-1$ for $r > 0$ and $s \leq 2n$. 
This proves \eqref{return time bd} with $c_1 = \frac{1}{1-p}$ and $c_2 = \frac{1}{2n} \log (\frac{1}{1-p})$.
\end{proof}

To conclude this subsection, we prove Proposition~\ref{cor form line}. The proof uses the fact that, if $C$ has $n \leq 2m+1$ elements, then $\C_2$ forms a line with at least a probability that depends on the diameter of $C$.

\begin{proposition}\label{line wpp}
There are positive numbers $c_1, c_2$ such that, if $C \subset_n \Z^d$ has $n \in \llb m+1,2m+1\rrb$ elements, then
\[
\P_{C} ( \C_2 \in \cal L ) \geq c_1 \diam (C)^{-c_2}.
\]
In fact, the same bound applies for any $n > m$ with $\C_{\frac{n-1}{m}}$ in the place of $\C_2$ if $n-1$ is a multiple of $m$ and $\C_{[\frac{n-1}{m}]+1}$ otherwise.
\end{proposition}

\begin{proof}
Let $x \in C$ be an element with the greatest $e_1$ component among elements in $C$. In the first step, activate any $m$ elements other than $x$, and transport them to $x+e_1, \dots, x + me_1$. In the second step, activate the remaining $n-(m+1) \leq m$ elements, along with $m - (n - (m+1))$ that were activated in the first step. Place the former elements at $x - e_1, \dots, x - (n-(m+1))e_1$, and keep the latter elements where they are. When these steps occur,
\[
\C_2 = \{x - (n - (m+1))e_1, \dots, x + m e_1\} \in \cal L.
\]
CAT realizes $\cal L$ in this way with positive probability. The claimed bound then follows from Lemma~\ref{poss implies typ}. We omit the generalization to $n > m$, which is straightforward.
\end{proof}

Proposition~\ref{cor form line} is a simple consequence of Theorem~\ref{drift} and Proposition~\ref{line wpp}.

\begin{proof}[Proof of Proposition~\ref{cor form line}]
We write
\[
\P_{C} (\C_{n+2} \in \cal L) \geq \P_{C} (\diam (\C_n) \leq 2n^2) \, \P_{C} (\C_{n+2} \in \cal L \mid \diam (\C_n) \leq 2n^2).
\]
The first factor is at least $q_1>0$ by Theorem~\ref{drift}, while the second factor is at least $q_2>0$ by the Markov property applied at time $n$ and Proposition~\ref{line wpp}. Taking $p = q_1q_2$ proves the claim.
\end{proof}

\subsection{Growth of the norm of the random walk}

In this subsection, we work with an arbitrary pair of distinct clusters. For simplicity, we label these clusters $1$ and $2$, and we omit the $ij$ subscripts from all notation. For example, we denote $S_k^{12}$ as $S_k$. Throughout this subsection, we assume that $D^I \in \Tup$ satisfies \eqref{stable def}, in which case $(S_k)_{k \geq 0}$ under $\P_{D^I}$ is a symmetric, aperiodic, irreducible random walk on $\Z^d$, by Proposition~\ref{wm is rw}. 

We need two standard random walk estimates. Denote the first time that $(S_k)_{k \geq 0}$ enters $A \subseteq \Z^d$ by $T_A := \inf \{k \geq 0: S_k \in A \}$. Propositions 2.4.5 and 6.4.2 of \cite{lawler2010} state that there are $\alpha_3, \alpha_4, \alpha_5 \in \R_{>0}$ such that, for any $R, \lambda \in \R_{>0}$ and $x \in B(R)^c$, 
\begin{align}
\P_x (T_{B(R)^c} > \lambda R^2) &\leq \alpha_3 e^{-\alpha_4 \lambda}, \label{rw est1}\\ 
\P_x (T_{B(R)} < \infty) &\leq \alpha_5 \left( \frac{R}{\| x \|} \right)^{d-2}. \label{rw est2}
\end{align}

\begin{remark}[The estimates hold even though the increments are unbounded]
The statements of Propositions 2.4.5 and 6.4.2 of \cite{lawler2010} assume that the increments of the random walk are bounded, which is not true of $S_k$. However, the proofs of the propositions apply as written because the increments of $S_k$ are exponentially tight, by Proposition~\ref{exp tail}. The only exception is that the appeal to Proposition 4.3.1 in the proof of Proposition 6.4.2 must be replaced by one to Proposition 4.3.5, with the same conclusion.
\end{remark}

\begin{remark}[The same constants work for every pair of clusters]
In general, $\alpha_3$, $\alpha_4$, and $\alpha_5$ depend on the increment distribution of the random walk. However, there are at most $(m+1)^2$ increment distributions among the walks indexed by $i$ and $j$, since each cluster has a number of elements in $\llbracket m+1, 2m+1 \rrbracket$. We can therefore choose the constants in \eqref{rw est1} and \eqref{rw est2} to hold for every pair of clusters.
\end{remark}

Let $r > 1$ and define
\begin{equation}\label{w and d}
w(r) = \left[ \alpha_2^{-1} \log \left( 2 \alpha_1 ( \log r)^2 \right) (e r)^2 \right] + 1 \quad \text{and} \quad \delta (r) = \frac{1}{2 \alpha_3 (\log r)^2}.
\end{equation}

\begin{proposition}\label{doubling est}
If $x \in \Z^d$ satisfies $r:= \|x\| > 1$, then
\begin{equation*}
\P_x \left( T_{B(e r)^c} \leq w(r), \,\, T_{B(\delta(r) r)} = \infty \right) \geq 1 - \frac{1}{(\log r)^2}.
\end{equation*}
\end{proposition}

\begin{proof}
By a union bound, it suffices to show that $\P_x (T_{B(e r)^c} > w( r ))$ and $\P_x (T_{B(\delta( r ) r)} < \infty)$ are each at most $\frac{1}{2(\log r)^2}$. These bounds follow from substituting $w(r)$ and $\delta (r)$ into \eqref{rw est1} and \eqref{rw est2}.
\end{proof}

Define the random times
\begin{align}
\eta_l (r) &:= \inf \{k \geq \eta_{l-1} (r): \| S_k \| > e \| S_{\eta_{l-1} (r)} \| \} \label{eta def}\\ 
\theta_l (r) &:= \inf \{k \geq \eta_{l-1} (r): \| S_k \| \leq \delta ( S_{\eta_{l-1} (r)} ) \| S_{\eta_{l-1} (r)} \| \}. \nonumber
\end{align}
Additionally, define the sequence of events $(G_l (r))_{l \geq 1}$ according to
\[
G_l (r) := \left\{ \eta_l (r) - \eta_{l-1} (r) \leq w \left( \| S_{\eta_{l-1} (r)} \| \right), \,\, \theta_l (r) = \infty \right\}.
\]
In these terms, Proposition~\ref{doubling est} states that
\begin{equation*}
\forall x \in \Z^d: \|x\|>1, \quad \P_x (G_1 ( \|x\|)) \geq 1 - \frac{1}{(\log \| x\|)^2}.
\end{equation*}

\begin{proposition}\label{prop g is typical}
Let $x \in \Z^d$ satisfy $r:= \|x\| > 1$. Then
\begin{equation}\label{g is typical}
\P_x \left( \cup_{l=1}^\infty G_l (r)^c \right) \lesssim \frac{1}{\log r}.
\end{equation}
\end{proposition}

\begin{proof}
Denote $H_l (r) = \cap_{k \leq l} G_k (r)$. We will show that
\begin{equation}\label{qj bd}
q_j := \P_x \left( G_{l+1} (r)^c \bigm\vert H_l (r) \right) \leq \frac{1}{(\log (e^l r))^2},
\end{equation}
which implies \eqref{g is typical} because
\[
\P_x \left( \cup_{l=1}^\infty G_l (r)^c \right) \leq \sum_{l=0}^\infty q_l \lesssim \frac{1}{\log r}.
\]

To establish \eqref{qj bd}, we use the tower rule with the $\sigma$-field generated by $(S_0, \dots, S_{\eta_l (r)})$ and apply the strong Markov property to $\eta_l (r)$ as
\begin{equation}\label{qj tower}
\P_x \left( G_{l+1} (r) \bigm\vert H_l (r) \right) = \E_x \left[ \P_{S_{\eta_l (r)}} \left(G_1 \big( \| S_{\eta_l (r)} \| \big)\right) \frac{\1_{H_{l} (r)} }{\P_x (H_{l} (r))} \right].
\end{equation}
When $H_l (r)$ occurs, $\| S_{\eta_l (r)} \| > e^l r$, hence Proposition~\ref{doubling est} implies
\[
\P_{S_{\eta_l (r)}} \left(G_1 \big( \| S_{\eta_l (r)} \| \big) \right) \geq 1 - \frac{1}{(\log (e^l r))^2}. 
\]
Substituting this bound into \eqref{qj tower} gives \eqref{qj bd}.
\end{proof}

The next proposition controls the absolute and relative separation of a pair of clusters when 
\begin{equation}\label{fij gij}
\cal F := \cap_{l \geq 1} F_l \quad \text{and} \quad \cal G (r) := \cap_{l \geq 1} G_l (r)
\end{equation}
occur for sufficiently large $r>0$. We state it in terms of $R_l := \| S_{\eta_l} \|$, the norm of the random walk $l \in \Z_{\geq 0}$ ``doublings.''\footnote{Recall that $\|S_{\eta_l}\|$ is the norm of the random walk when it first exceeds $e\|S_{\eta_{l-1}}\|$ \eqref{eta def}. In particular, $R_l$ is at least $e^l R_0$.}

\begin{proposition}[Time, diameter, and separation bounds]\label{dist rel to diam} 
Let $l \in \Z_{\geq 0}$. If $R_0$ is sufficiently large, then the occurrence of $\cal F \cap \cal G (R_0)$ implies that
\begin{equation}\label{dist and diam bds}
t \lesssim (R_l \log R_l)^2, \quad \diam (\D_t^1) \lesssim \log R_l, \quad 
\dist (\D_t^1, \D_t^2) \gtrsim \frac{R_l}{(\log R_l)^2},
\end{equation}
for all $t \in N^{-1} \big( [\eta_l, \eta_{l+1}) \big)$.
\end{proposition}

The first two bounds are consequences of the definitions of $\cal F$ and $\cal G (r)$, while the third further requires Proposition~\ref{lem dist bd} to relate cluster separation to $R_l$. 
The tricky aspect of Proposition~\ref{dist rel to diam} is that it features three different counts of time: $t$ counts the steps of $\D_t$, $k$ counts the steps of the derived random walk $S_k$, and $l$ counts the doublings of $\|S_{\eta_l}\|$. To reconcile these counts---in particular, to prove the first bound of \eqref{dist and diam bds}---we will obtain estimates for $t$ in terms of $\xi_k$ and $\eta_l$; for $\xi_k$ in terms of $\eta_l$; and for $\eta_l$ in terms of $R_l$. 
In the proof, we will use the notation $f \lesssim_{R_0} g$ to indicate that $f \lesssim g$ holds under the assumption that $R_0$ is sufficiently large.

\begin{proof}[Proof of Proposition~\ref{dist rel to diam}]
We can bound $t$, $\diam (\D_t^1)$, and $\dist (\D_t^1, \D_t^2)$ in terms of $\xi_k$ and $\eta_l$ in the following way:
\begin{align}
t & \leq \sum_{0 \leq i \leq l} \sum_{j = \eta_i}^{\eta_{i+1}-1} (\xi_{j+1} - \xi_j), \label{t bd1}\\ 
\diam (\D_t^1) & \leq \xi_{N(t)+1} - \xi_{N(t)}, \quad \text{and} \label{diam bd1}\\ 
\dist (\D_t^1, \D_t^2) & \geq \|S_{N(t)}\| - O \left( (\xi_{N(t)+1} - \xi_{N(t)})^2 \right).\label{dist bd1}
\end{align}
The first bound holds because, by assumption, there have not been $l+1$ doublings by time $t$. The second bound holds because the diameter of $\D_{\xi_{N(t)}}^1$ is at most $n$, and it grows at most linearly in time (a consequence of \eqref{amlg prop} and \eqref{copycat def}) until $t \leq \xi_{N(t)+1}$. The third bound is Proposition~\ref{lem dist bd}.

We use the occurrence of $\cal F \cap \cal G (R_0)$ to bound $\eta_l$ in terms of $R_l$, and $\xi_k$ in terms of $\eta_l$. By definition, when $\cal F \cap \cal G (R_0)$ occurs, we have
\begin{align}
\forall l \geq 0, \quad &\eta_{l+1} - \eta_l \leq w(R_l), \label{eta diff bd1}\\ 
\forall k \geq 0, \quad &\xi_{k+1} - \xi_k \leq \alpha_1 \log (\alpha_2 (k+1)), \quad \text{and} \label{xi diff bd1}\\ 
\forall k \in [\eta_l, \eta_{l+1}), \quad &\|S_k\| \geq \delta (R_l) R_l. \label{sk bd1}
\end{align}
Since $w(s)$ grows more quickly than $s^2$ \eqref{w and d}, it is easy to see that
\begin{equation}
\eta_{l+1} = \sum_{0 \leq j \leq l} (\eta_{j+1} - \eta_j) \stackrel{\eqref{eta diff bd1}}{\leq} \sum_{0 \leq j \leq l} w(R_j) \lesssim_{R_0} w(R_l).\label{eta wr}
\end{equation}
Consequently, if $k \in [\eta_l,\eta_{l+1})$, then 
\begin{equation}
\xi_{k+1} - \xi_k \stackrel{\eqref{xi diff bd1}}{\leq} \alpha_1 \log (\alpha_2 \eta_{l+1}) \stackrel{\eqref{eta wr}}{\lesssim_{R_0}} \log w (R_l).\label{xi diff bd}
\end{equation}
Lastly, by definition \eqref{w and d},
\begin{equation}\label{w and d bds}
w(R_l) \lesssim_{R_0} R_l^2 \log R_l \quad \text{and} \quad \delta (R_l) \gtrsim (\log R_l)^{-2}.
\end{equation}
In view of \eqref{t bd1}, the three preceding bounds show that $t$ is at most a sum of roughly $R_l^2 \log R_l$ summands, each of which is no more than roughly $\log R_l$. More precisely,  we have
\[
t \stackrel{\eqref{t bd1}}{\leq}  \sum_{0 \leq i \leq l} \sum_{j=\eta_i}^{\eta_{i+1} - 1} (\xi_{j+1}-\xi_j) 
 \stackrel{\eqref{eta wr}, \,\,\eqref{xi diff bd}}{\lesssim_{R_0}} w(R_l) \log w(R_l) \stackrel{\eqref{w and d bds}}{\lesssim_{R_0}} (R_l \log R_l)^2.
\]
This proves the first bound of \eqref{dist and diam bds}. The other two bounds follow from:
\begin{align*}
\diam (\D_t^1) &\stackrel{\eqref{diam bd1}}{\leq} \xi_{N(t)+1} - \xi_{N(t)} \stackrel{\eqref{xi diff bd}}{\lesssim_{R_0}} \log w (R_l) \stackrel{\eqref{w and d bds}}{\lesssim_{R_0}} \log R_l,\\ 
\dist (\D_t^1, \D_t^2) &\stackrel{\eqref{dist bd1}}{\geq} \|S_{N(t)}\| - O \left( ( \xi_{N(t)+1} - \xi_{N(t)})^2 \right)\\ 
&\stackrel{\eqref{sk bd1}}{\geq} \delta (R_l) R_l - O \left( (\log w (R_l))^2 \right) \stackrel{\eqref{w and d bds}}{\gtrsim_{R_0}} R_l (\log R_l)^{-2}.
\end{align*}
\end{proof}

\subsection{Proof of Proposition~\ref{d growth}}

We now consider all distinct pairs of clusters and use the results of the previous section to prove that the overall separation of clusters grows, in the sense of Proposition~\ref{d growth}. Accordingly, we reintroduce the $ij$ superscripts and define
\[
\cal F = \cap_{i,j \in I: \, i < j} \cal F^{ij} \quad \text{and} \quad \cal G (r) = \cap_{i,j \in I: \, i < j} \cal G^{ij} (r)
\]
for $r > 0$, where $\cal F^{ij}$ and $\cal G^{ij} (r)$ are the events in \eqref{fij gij} that we previously denoted by $\cal F$ and $\cal G (r)$. Additionally, we define $R_0 = \min_{i, j \in I: \, i < j} R_0^{ij}$, where $R_0^{ij} = \| S_0^{ij} \|$.

The next two results show that there are $b, \gamma > 0$ such that, for any $\delta \in (0, \frac12)$, if $R_0$ is sufficiently large, then
\begin{equation}\label{fgro incl}
\{ \cal F \cap \cal G (R_0)\} \subseteq \left\{(\D_t^I)_{t \geq 0} \in \Tup_{\tilde R_0,b}^\N\right\} \cap \left\{\sep (\D_t^I) \geq \gamma t^{\frac12 - \delta}\right\} = \left\{(\D_t^I)_{t \geq 0} \in \Growth_{\tilde R_0, b, \gamma, \delta} \right\},
\end{equation}
where $\tilde R_0 = R_0^{0.99}$. We need to use $\tilde R_0$ in the place of $R_0$ because the definition of $\cal G (R_0)$ permits the separation to temporarily drop to $\delta (R_0) R_0$ (see the definitions of $\theta_l (r)$ \eqref{eta def} and $G_l (r)$). Both of the proofs are applications of Proposition~\ref{dist rel to diam}. 

\begin{proposition}\label{fng implies sep}
There is $b>0$ such that, if $R_0$ is sufficiently large, then 
\[
\{\cal F \cap \cal G (R_0)\} \subseteq \left\{ (\D_t^I)_{t \geq 0} \in \Tup_{\tilde R_0,b}^\N\right\},
\]
where $\tilde R_0 = R_0^{0.99}$.
\end{proposition}

\begin{proof}
For $(\D_t^I)_{t \geq 0}$ to belong to $\Tup_{r,b}^\N$, (i) the separation of $\D_t^I$ must always be at least $r$ and (ii) $\diam (\D_t^i) \leq b \log \dist (\D_t^i,\D_t^{\neq i})$ must hold for every $t \geq 0$ and $i \in I$. 
By Proposition~\ref{dist rel to diam}, if $\cal F \cap \cal G (R_0)$ occurs for sufficiently large $R_0$, then
\[
\dist (\D_t^i,\D_t^{\neq i}) \gtrsim R_l(\log R_l)^{-2} \gtrsim R_0 (\log R_0)^{-2} \quad \text{and} \quad \diam (\D_t^i) \lesssim \log \dist (\D_t^i,\D_t^{\neq i}).
\]
The first bound implies that (i) holds for large enough $b$, while the second implies that (ii) holds if $R_0$ is sufficiently large.
\end{proof}

\begin{proposition}\label{fng implies finite sum}
There is $\gamma>0$ such that, for any $\delta \in (0, \frac12)$, if $R_0$ is sufficiently large, then
\[
\{\cal F \cap \cal G (R_0)\} \subseteq \left\{\sep (\D_t^I) \geq \gamma t^{\frac12 - \delta}\right\}.
\]
\end{proposition}

\begin{proof}
Fix a pair $i, j \in I$ such that $i < j$ and $l \in \Z_{\geq 0}$. It suffices to prove that there is $\gamma>0$ such that, for any $\delta \in (0,\frac12)$, if $R_0^{ij}$ is sufficiently large, then the occurrence of $\cal F^{ij} \cap \cal G^{ij} (R_0^{ij})$ implies that
\begin{equation*}
\dist (\D_t^i, \D_t^j) \geq \gamma t^{\frac12 - \delta},
\end{equation*}
for every $t \in N^{-1} \big( [\eta_l^{ij}, \eta_{l+1}^{ij}) \big)$ and $\delta \in (0, \frac12)$.  
By Proposition~\ref{dist rel to diam}, if $R_0^{ij}$ is sufficiently large, then the occurrence of $\cal F^{ij} \cap \cal G^{ij} (R_0^{ij})$ implies that
\begin{equation}\label{dij bd}
t \lesssim (R_l^{ij} \log R_l^{ij})^2 \quad \text{and} \quad 
\dist (\D_t^i, \D_t^j) \gtrsim R_l^{ij} (\log R_l^{ij})^{-2}.
\end{equation} 
Hence, 
\[
\frac{\dist (\D_t^i, \D_t^j)}{t^{\frac12 - \delta}} \gtrsim \frac{(R_l^{ij})^{2\delta}}{(\log R_l^{ij})^{3-2\delta}} \geq \frac{(R_0^{ij})^{2\delta}}{(\log R_0^{ij})^{3-2\delta}}.
\]
If $R_0^{ij}$ is sufficiently large in terms of $\delta$ to satisfy $R_0^{2\delta} \geq (\log R_0)^{3-2\delta}$, then the preceding bound shows that $\dist (\D_t^i, \D_t^j) \gtrsim t^{1/2-\delta}$. 
\end{proof}

The two preceding propositions establish the inclusion \eqref{fgro incl}. To prove Proposition~\ref{d growth}, it remains to show that $\cal F \cap \cal G (R_0)$ typically occurs when $R_0$ is sufficiently large. The proof combines Propositions~\ref{prop short rets} and \ref{prop g is typical} with a union bound.

\begin{proposition}\label{fng is typical}
Let $D^I \in \Tup$ satisfy \eqref{stable def}. If $R_0$ is sufficiently large, then
\begin{equation*}
\P_{D^I} ( \cal F \cap \cal G (R_0)) \geq \frac14.
\end{equation*}
\end{proposition}

\begin{proof}
First, Proposition~\ref{prop g is typical} states that, for every pair of clusters $i,j \in I$ such that $i < j$, the event $\cal G^{ij} (R_0^{ij})$ occurs except with a probability of at most $O \big( (\log R_0^{ij} )^{-1} \big)$. By a union bound over the at most $n^2$ pairs of clusters, $\cal G (R_0)$ occurs except with a probability of at most $O \big( n^2 (\log R_0)^{-1} \big)$. Assume that $R_0$ is sufficiently large to make this bound $\frac34$. Second, by Proposition~\ref{prop short rets}, if $D^I$ satisfies \eqref{stable def}, then $\P_{D^I} (\cal F) \geq \frac12$. The two bounds together imply that $\P_{D^I} (\cal F \cap \cal G (R_0)) \geq \frac14 > 0$.
\end{proof}

\begin{proof}[Proof of Proposition~\ref{d growth}]
Let $a \in \Z_{\geq 0}$ and $C \in \Stable_a$. By definition \eqref{reach well sep conds}, there is a partition of $C$ into $D^I \in \Tup$ that satisfies \eqref{stable def} and the associated random walks of which are initially separated by $R_0 \geq a$. Hence, by Propositions~\ref{fng implies sep} and \ref{fng implies finite sum}, there are $b, \gamma > 0$ such that, for any $\delta \in (0, \frac12)$, if $a$ is sufficiently large, then 
\[
\{ \cal F \cap \cal G (R_0) \} \stackrel{\eqref{fgro incl}}{\subseteq} \left\{ (\D_t^I)_{t \geq 0} \in \Growth_{\tilde R_0,b,\gamma,\delta} \right\} \subseteq \left\{ (\D_t^I)_{t \geq 0} \in \Growth_{\tilde a,b,\gamma,\delta} \right\},
\]
where $\tilde R_0 = R_0^{0.99}$ and $\tilde a = a^{0.99}$. The claimed bound then follows from Proposition~\ref{fng is typical}.
\end{proof}

\section{Cluster formation}\label{sec cluster formation}

In this section, we prove Theorem~\ref{reach well sep}, which concerns the formation of well separated clusters that can persist \eqref{reach well sep conds}. We do this in two stages. First, we use Lemma~\ref{depl small comps} to group elements into connected components with more than $m$ elements. Second, we partition these connected components into parts with between $m+1$ and $2m+1$ elements, before ``treadmilling'' them apart, one by one.

\subsection{Staying lemma}

We need a basic fact about CAT. If $U \subset_n \Z^d$ satisfies
\begin{equation}
\exists V = \{v_1, \dots, v_m\} \subset_m U: \quad \diam (V) \leq n \quad \text{and} \quad v_l \in \bdy \big( U \setminus \{v_l, \dots, v_m\} \big), \,\, 1 \leq l \leq m,
\label{all a seq bdy}
\end{equation}
then it is typical for CAT to ``stay'' at $U$, since it can activate the elements at $v_1, \dots, v_m$ and then simply return them. To state the next result, for a collection $\cal U$ of $n$-element subsets of $\Z^d$, define $T_{\cal U} := \inf \{t \geq 0: \C_t \in \cal U\}$.

\begin{lemma}[Staying lemma]\label{staying lem 2}
Let $\cal U$ be a collection of $n$-element subsets of $\Z^d$. If every $U \in \cal U$ satisfies \eqref{all a seq bdy}, and if there are $t \in \Z_{\geq 0}$ and $q > 0$ such that $\P_{C} (T_{\cal U} \leq t) \geq q$  for every $C \subset_n \Z^d$, then there is $p>0$ such that
\[
\P_{C} (\C_t \in \cal U) \geq p^t
\] 
for every $C \subset_n \Z^d$.
\end{lemma}

\begin{proof}
Since $\P_C (T_{\cal U} \leq t) \geq q$ and $\{\C_t \in \cal U, T_{\cal U} \leq t\} \subseteq \{\C_{T_{\cal U}} = \cdots = \C_t, T_{\cal U} \leq t \}$, we have
\[
\P_{C} (\C_t \in \cal U) \geq q \P_{C} (\C_t \in \cal U \mid T_{\cal U} \leq t) \geq q \prod_{s=T_{\cal U}+1}^t \P_C ( \C_s = \C_{s-1} \mid \C_{T_{\cal U}} = \cdots \C_{s-1}, T_{\cal U} \leq t),
\]
for every $C \subset_n \Z^d$. 
By Lemma~\ref{poss implies typ}, there is $r > 0$ such that $\P_U (\C_1 = \C_0) \geq r$ 
for every $U$ that satisfies \eqref{all a seq bdy}. Hence, by the strong Markov property applied to $s \in \llb T_{\cal U} + 1, t \rrb$ and the fact that $\C_{s-1}$ satisfies \eqref{all a seq bdy} given that $\{\C_{T_{\cal U}} = \cdots = \C_{s-1}, T_{\cal U} \leq t \}$ occurs, we have
\[
\P_C ( \C_s = \C_{s-1} \mid \C_{T_{\cal U}} = \cdots \C_{s-1}, T_{\cal U} \leq t) \geq r,
\]
which implies the claimed bound with $p = qr$.
\end{proof}

\subsection{Application of Lemma~\ref{depl small comps}}

Lemma~\ref{depl small comps} states that, if some element of $\C_0$ contains an element that belongs to a connected components with $m$ or fewer elements, then $\C_1$ typically has fewer such elements. The next result states that, consequently, $\C_n$ typically has no such elements. In other words, the event
\[
\NoSmallComps_t := \cap_{x \in \C_t} \{|\comp_{\C_t} (x)| > m\}
\]
typically occurs for $t=n$. (We continue to use $\comp_C (y)$ to denote the set of elements that are connected to $y \in \Z^d$ in $C \subset \Z^d$.)

\begin{proposition}\label{form big comps}
There is $p>0$ such that
\begin{equation}\label{eq form conn comp}
\P_{C} (\NoSmallComps_n) \geq p,
\end{equation}
for every $C \subset_n \Z^d$.
\end{proposition}

\begin{proof} 
Define $\tau$ to be the first time $t \geq 0$ that $\NoSmallComps_t$ occurs. Recall the event $\DepleteSmallComps_t$, which is defined as $\{|\scr R_{t+1}| < |\scr R_t|\}$ in terms of $\scr R_t$, the set of $x \in \C_t$ that belong connected components of $\C_t$ with $m$ or fewer elements \eqref{dsc def}. The key observation is that
\begin{equation}\label{tau deplete}
\P_{C} (\tau \leq n) \geq \P_{C} (\cap_{t < \tau} \mathsf{DepleteSmallComps}_t).
\end{equation} 
Indeed, since $|\scr R_0| \leq n$, $|\scr R_t|$ can decrease at most $n$ times consecutively before time $\tau$. If $\{\tau \leq n\}$ occurs, then we can realize $\NoSmallComps_n$ by simply staying at $\C_\tau$ until time $n$.

Lemma~\ref{depl small comps} implies that there is $q_1 > 0$ such that
\[
\P_{C} (\cap_{t < \tau} \mathsf{DepleteSmallComps}_t) \geq q_1,
\]
for every $C \subset_n \Z^d$, while Lemma~\ref{staying lem 2} implies that there is $q_2 > 0$ such that
\[
\P_{C} (\NoSmallComps_n \mid \tau \leq n) \geq q_2.
\]
Indeed, Lemma~\ref{staying lem 2} applies to the collection $\cal U$ of $n$-element subsets of $\Z^d$, the elements of which exclusively belong to connected components with more than $m$ elements. This collection clearly satisfies \eqref{all a seq bdy}. Combining these bounds with \eqref{tau deplete} proves \eqref{eq form conn comp} with $p = q_1 q_2$.
\end{proof}

\subsection{A property of sets that have no small connected components}

The virtue of a set that consists exclusively of connected components with more than $m$ elements is that these components can be partitioned into parts that can be formed into lines and then ``treadmilled,'' one by one (Figures~\ref{nscFig} and \ref{treadFig}).

\begin{figure}
\includegraphics[width=0.8\textwidth]{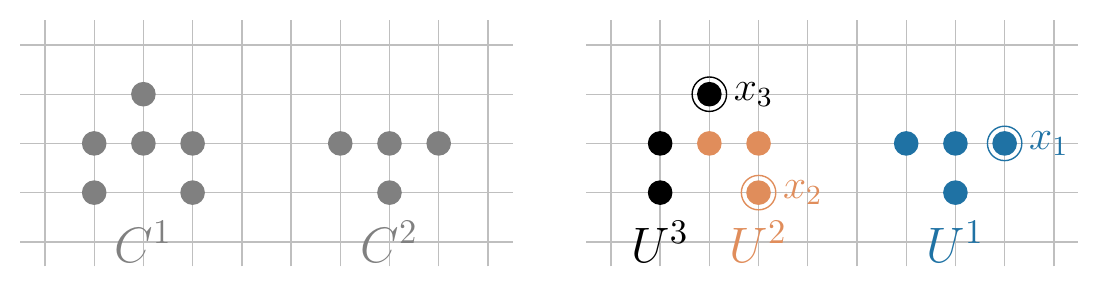}
\caption{An example of Lemma~\ref{imp of nosmallcomp} with $(d,m,n) = (2,2,10)$. When a set $C$ (left) has connected components ($C^1$ and $C^2$) with more than $m$ elements each, they can be partitioned into $(U^j)_j$, in such a way that each $U^j$ has between $m+1$ and $2m+1$ elements and contains an element $x_j$ (circled) that (i) abuts a ray which is empty of $U_k$ and (ii) is at least as far to the right as $x_k$, for each $k \geq j$. These properties enable the parts to be treadmilled to the right, one by one (Figure~\ref{treadFig}).}
\label{nscFig}
\end{figure}

\begin{lemma}\label{imp of nosmallcomp}
If every element of a configuration $C \subset_n \Z^d$ belongs to a connected component with more than $m$ elements, then there is a partition of $C$ into $(U^j)_{j=1}^J$, each part of which satisfies $|U^j| \in \llbracket m+1, 2m+1 \rrbracket$ and $\diam (U^j) \leq n$. Moreover, each $U^j$ contains some element $x_j$ for which the ray $(x_j+L_{1,\infty})$ is empty of $\cup_{k \geq j} U^k$ and $x_j \cdot e_1 \leq x_i \cdot e_1$, for every $i < j$.
\end{lemma}

\begin{proof}
Denote the connected components of $C$ by $(C^i)_{i \in I}$. Note that, since $|C^i| \geq m+1$, there is a partition $(n_l^i)_l$ of $|C^i|$ such that $n_l^i \in \llbracket m+1, 2m+1 \rrbracket$. To form $U^j$, if $x_j \in C^{i_j} \setminus \{U^k\}_{k<j}$ has the greatest $e_1$ component among elements of $C \setminus \{U^k\}_{k<j}$---and if we have allocated elements from component $i_j$ in exactly $l_j$ previous rounds---then we allocate $x_j$ and any other $n_{l_j+1}^{i_j} - 1$ elements of $C^{i_j}\setminus \{U^k\}_{k<j}$ to $U^j$. By construction, $(U^j)_j$ has the claimed properties. 
\end{proof}

\subsection{Treadmilling lemma}

The next result states that, if $C$ contains a line segment of $m+1$ or more elements, next to an unoccupied line segment, then it is typical for the former segment to be translated into the latter. We refer to this as {\em treadmilling} because it can be realized by activation at the ``tail'' of the segment, followed by transport to its ``head.'' Recall that, for integers $i\leq j$, $L_{i,j}$ denotes the line segment $\{ke_1: i \leq k \leq j\}$. For example, if $x \in \Z^d$, then $(x+L_{0,k-1})$ is the line segment of length $k$, with endpoints at $x$ and $x+(k-1)e_1$. 

\begin{figure}
\includegraphics[width=0.8\textwidth]{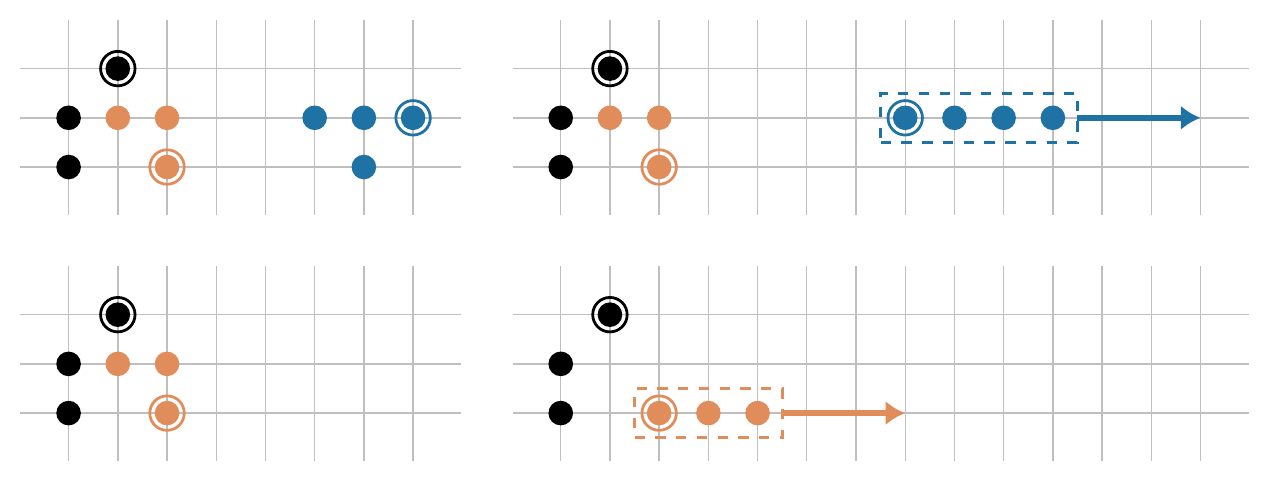}
\caption{Continuation of Figure~\ref{nscFig}. If a set has no small connected components, then these components can be partitioned into parts that can be treadmilled to the right, one at a time, using Lemma~\ref{lem treadmill 2}.}
\label{treadFig}
\end{figure}

\begin{lemma}[First treadmilling lemma]\label{lem treadmill 1}
Suppose that $C \subset_n \Z^d$ contains $(x+L_{0,k-1})$ for some $x \in \Z^d$ and $k \geq m+1$. There is $p>0$ such that, if, for some $l \in \llbracket 1,m \rrbracket$, the segment $(x+L_{k,k+l-1})$ is disjoint from $C$, then
\begin{equation}\label{eq treadmill 1}
\P_{C} \left( \C_1 = \big(C \setminus (x+L_{0,k-1}) \big) \cup (x+L_{l,k+l-1}) \right) \geq p.
\end{equation}
\end{lemma}

\begin{proof}
To treadmill $L_{0,k-1}$ a number of sites $l \leq m$ in the $e_1$ direction, we activate the ``tail,'' $(x+L_{0,m-1})$ and transport $l$ elements to the ``head,'' $(x+L_{l,k+l-1})$. We return the remaining $m-l$ elements back to $(x+L_{l,m-1})$. Hence, the probability in question is at least
\[
\P_{C} \left( \AA_0 = (x+L_{0,m-1}), \T_0 = (x+L_{l,m-1}) \cup (x+L_{k,k+l-1}) \right) > 0.
\]
In fact, since the diameter of $(x+L_{0,k+l-1})$ is at most $2n$, Lemma~\ref{poss implies typ} implies a lower bound of $p>0$, depending on $n$ only.
\end{proof}

The second treadmilling lemma forms a subset of $m+1$ or more elements of $C$ into a line segment and then repeatedly applies Lemma~\ref{lem treadmill 1} to treadmill this segment a desired distance (Figure~\ref{treadFig}).

\begin{lemma}[Second treadmilling lemma]\label{lem treadmill 2}
Let $r \in \Z_{\geq 2m}$ and $C \subset_n \Z^d$. There is $p>0$ such that, if $U \subseteq C$ has a diameter of at most $n$, satisfies $|U| \in \llbracket m+1,2m+1 \rrbracket$, and contains an $x$ for which $(x+L_{1,r})$ is disjoint from $C$, 
then 
\begin{equation}\label{eq treadmill 2}
\P_{C} (\C_{r+2} = (C \setminus U) \cup (x+L_{r-|U|+1,r}) ) \geq p^r.
\end{equation}
\end{lemma}

\begin{proof}
First, we form $U$ into the line segment $(x+L_{0,|U|-1})$ using two steps. Specifically, we claim that there is $q_1 > 0$ such that
\begin{equation}\label{tm2 bd1}
\P_{C} \big( \C_2 = (C \setminus U) \cup (x+L_{0,|U|-1}) \big) \geq q_1.
\end{equation} 
Indeed, we use the first step to activate any $m$ elements of $U \setminus \{x\}$ and transport them to $(x+L_{1,m})$. In the second step, we activate the other $|U|-(m+1)$ elements along with $(x+L_{m - (|U|-(m+1))+1,m})$---so that we activate $m$ elements in total---and transport the elements to $(x+L_{m - (|U|-(m+1))+1,|U|-1})$. (Note that, if $|U| = m+1$, then the second step leaves $\C_1$ unchanged.) 
This shows that the probability in \eqref{tm2 bd1} is positive. The bound then follows from Lemma~\ref{poss implies typ} because the preceding steps entail $\diam (\AA_j \cup \T_j) \leq 2n$ for $j \in \{0,1\}$.

Second, we treadmill the line segment $(x+L_{0,|U|-1})$ to $(x+L_{r-|U|+1,r})$ in $s = \lceil \frac{r - |U| + 1}{m} \rceil \leq r$ additional steps, where $\lceil \cdot \rceil$ denotes the ceiling function. By the Markov property applied to each time in $\llbracket 2, s+1\rrbracket$ and Lemma~\ref{lem treadmill 1}, we find that
\begin{equation}\label{tm2 bd2}
\P_{C} \big( \C_{s+2} = (C \setminus U) \cup (x+L_{r-|U|+1,r}) \bigm\vert \C_2 = D \big) \geq q_2^s,
\end{equation}
where $D$ denotes the configuration in \eqref{tm2 bd1} and $q_2$ is positive.

To conclude, we note that if $\cal U = \{(C\setminus U) \cup (x + L_{r-|U|+1,r})\}$, then \eqref{tm2 bd2} shows that $\{T_{\cal U} \leq r+2\}$ occurs with a probability of at least $q_2^s$. It is easy to verify that $\cal U$ satisfies the hypotheses of the staying lemma (Lemma~\ref{staying lem 2}), which implies that
\[
\P_{C} ( \C_{r+2} = (C\setminus U) \cup (x + L_{r-|U|+1,r}) \mid T_{\cal U} \leq r+2) \geq q_3^r,
\]
for $q_3 >0$, for every $C \subset_n \Z^d$. 
Combining the three preceding bounds yields
\[
\P_{C} \big( \C_{r+2} = (C \setminus U) \cup (x+L_{r-|U|+1,r} \big) \geq q_1 q_2^s q_3^r.
\]
Hence, \eqref{eq treadmill 2} holds with $p = q_1q_2q_3$.
\end{proof}

\subsection{Proof of Theorem~\ref{reach well sep}}

Following Figure~\ref{treadFig}, we will use Lemma~\ref{form big comps} to split $\C_0$ into connected components with more than $m$ elements, and then repeatedly apply Lemma~\ref{lem treadmill 2} to treadmill parts of each component in the $e_1$ direction, until they satisfy \eqref{reach well sep conds}. We will treadmill the parts in an order that ensures that each has an ``open lane.''

\begin{proof}[Proof of Theorem~\ref{reach well sep}]
Given $\C_n$, when $\NoSmallComps_n$ occurs, $\C_n$ has a partition into $(U^j)_{j=1}^J$ with the properties stated by Lemma~\ref{imp of nosmallcomp}. 
In terms of this partition, we define the event 
\[
\TreadPart_j = \left\{ \C_{s_j} = (\C_{s_{j-1}} \setminus U^j ) \cup V^j \right\},
\]
where
\begin{itemize}
\item $r_j = (r+n)(n-j)$ is how far we treadmill the $j$\textsuperscript{th} part;
\item $s_j = \sum_{i\leq j} (r_i + 2)$ is the time needed by Lemma~\ref{lem treadmill 2} to treadmill the first $j$ parts accordingly; and 
\item $V^j = (x_j + L_{r_j - |U^j| + 1, r_j})$ is the line segment that results from treadmilling the $j$\textsuperscript{th} part.
\end{itemize}

We claim that
\begin{align}
\P_{C} \left(\C_{2n^2r} \in \Stable_{r}\right) & \geq \E_{C} \left[ \P_{\C_n} \left( \C_{2n^2r-n} \in \Stable_{r} \right) \1_{\NoSmallComps_n} \right] \nonumber \\ 
&\geq \E_{C} \left[ \P_{\C_n} \left( \cap_{j\leq J} \TreadPart_j \cap \Stay \right) \1_{\NoSmallComps_n} \right], \label{stable to tread}
\end{align}
where $\Stay$ abbreviates $\cap_{t=s_J+1}^{2n^2r-n} \{\C_t = \C_{t-1}\}$.\footnote{A simple calculation shows that $r \geq 2n$ suffices to ensure that $s_J + 1 \leq 2n^2 r - n$ since $J \leq n$.} The first inequality follows from the tower rule and the Markov property applied to time $n$. The second inequality holds because, when $\cap_{j\leq J} \TreadPart_j$ and $\Stay$ occur, we have
\[
\C_{2n^2r-n} = \C_{s_J} = \cup_{j \leq J} V^j.
\]
When this event occurs, so too does $\{\C_{2n^2r-n} \in \Stable_{r}\}$ because $(V^j)_{j \leq J}$ satisfies \eqref{reach well sep conds}. Indeed, the properties of $(U^j)_{j \leq J}$ from Lemma~\ref{imp of nosmallcomp} include that $|U^j| \in \llbracket m+1,2m+1 \rrbracket$, hence the same is true of $|V^j|$. Additionally, $x_l \cdot e_1 \leq x_k \cdot e_1$ for $k < l$, which implies that  
\[
\dist (V^j, V^{\neq j}) \geq \min_{i \neq j} | r_i - |U_i| - r_j | \geq r.
\]

We bound below \eqref{stable to tread} with three lemmas; we verify their hypotheses at the end. First, Lemma~\ref{form big comps} states that
\begin{align}
\P_{C} (\NoSmallComps_n) &\geq q_1.\nonumber
\intertext{Second, by the Markov property applied at time $s_{j-1}$ and Lemma~\ref{lem treadmill 2}, we have}
\P_{\C_n} \left( \TreadPart_j \bigm\vert \cap_{i<j} \TreadPart_i \right) &\geq q_2^{r_j} \1_{\NoSmallComps_n}.\label{tread bound}
\intertext{Third, by Lemma~\ref{staying lem 2} and the trivial bound $2n^2r - n - s_J \leq 2n^2 r$, we have} 
\P_{\C_n} \left( \Stay \bigm\vert \cap_{j \leq J} \TreadPart_j \right) &\geq q_3^{2n^2r} \1_{\NoSmallComps_n}.\label{stay bound}
\end{align}
(Each of $q_1, q_2, q_3$ is a positive number.)

These bounds together imply that
\begin{equation}\label{eco1}
\E_{C} \left[ \P_{\C_n} \left( \cap_{j=1}^J \TreadPart_j \cap \Stay \right) \1_{\NoSmallComps_n} \right]
\geq q_1 q_2^{\sum_j r_j} q_3^{2n^2r},
\end{equation}
where the exponent on $q_2$ satisfies
\[
\sum_{j \leq J} r_j \leq n^2 (r+n) \leq 2n^2 r.
\]
The first inequality uses the fact that $J \leq n$ and $r_j \leq (r+n)n$; the second uses the assumption that $r \geq 2n$. By substituting this bound into \eqref{eco1}, and then \eqref{eco1} into \eqref{stable to tread}, we find that
\[
\P_{C} \left(\C_{2n^2r} \in \Stable_{r}\right) \geq q_1 (q_2 q_3)^{2n^2r}.
\]
In particular, \eqref{eq reach well sep} holds with $p = (q_1 q_2 q_3)^{2n^2}$.

It remains to verify the hypotheses of Lemma~\ref{lem treadmill 2} and Lemma~\ref{staying lem 2} that we used to obtain \eqref{tread bound} and \eqref{stay bound}.
\begin{itemize}
\item Only one hypothesis of Lemma~\ref{lem treadmill 2} is not obviously satisfied---namely, that $U^j$ contains an $x$ for which $(x+L_{1,r_j})$ is disjoint from $\C_{s_{j-1}}$. In fact, $x_j$ satisfies this requirement because, by Lemma~\ref{imp of nosmallcomp}, the ray $x_j + L_{1,\infty}$ is empty of $\{U^i\}_{i \geq j}$ and, when $\cap_{i<j} \TreadPart_i$ occurs, the rest of $\C_{s_{j-1}}$ consists of $\{V^i\}_{i<j}$, which is at least a distance $r_{j-1} - n > r_j$ away from $x_j$.
\item Our use of Lemma~\ref{staying lem 2} requires that, for each $t \in \llbracket s_J+1, 2nr^2 - n \rrbracket$, $\C_t$ contains an $m$-element subset that satisfies \eqref{seq boundary}. But every $V^j$ contains such a subset when $\cap_{j \leq J} \TreadPart_j$ and $\cap_{u = s_J+1}^{t-1} \{\C_u = \C_{u-1} \}$ occur. 
\end{itemize}
\end{proof}

\section{Irreducibility}\label{sec irred}

This section completes the proof of Theorem~\ref{rec to trans} by establishing that CAT is irreducible on $\wtProg_{m,n}$ (Proposition~\ref{irred}). Denote the first time that CAT reaches $U \subset_n \Z^d$ by $T_U := \inf \{t \geq 0: \C_t = U\}$ and further denote $T_{\wt U} := \inf \{t \geq 0: \wt \C_t = \wt U\}$. 
To be irreducible on $\wtProg_{m,n}$, CAT must satisfy
\begin{equation}\label{irred cond}
\forall U, V \in \Prog_{m,n}, \quad \P_{U} \big( T_{\wt V} < \infty \big) > 0.
\end{equation}
In other words, CAT can reach any such $\wt V$ from any such $U$. We prove this by showing that CAT can reach a line from any $U$, and it can reach any $\wt V$ from a line. We continue to denote $L_{1,\nn} = \{e_1, 2e_1, \dots, \nn e_1\}$.

\begin{proposition}[Set to line]\label{set to line}
For every $U \in \Prog_{m,n}$, 
\[
\P_U \big( T_{\wt{L}_{1,\nn}} < \infty \big) > 0.
\]
\end{proposition}

\begin{proposition}[Line to set]\label{line to set}
For every $V \in \Prog_{m,n}$,
\[
\P_{L_{1,\nn}} \big( T_V < \infty \big) > 0.
\]
\end{proposition}

\begin{proof}[Proof of Proposition~\ref{irred}]
The inequality \eqref{irred cond} follows from the strong Markov property applied to $T_{\wt L_{1,\nn}}$ and Propositions~\ref{set to line} and \ref{line to set}. Hence, CAT is irreducible on $\wtProg_{m,n}$.
\end{proof}

We turn our attention to proving Proposition~\ref{line to set}, as Proposition~\ref{set to line} is implied by an earlier result.

\begin{proof}[Proof of Proposition~\ref{set to line}]
This directly follows from Proposition~\ref{line wpp}.
\end{proof}

We will prove Proposition~\ref{line to set} by induction on $\kk$. To facilitate the proof, we preemptively address two cases. The next result addresses the first case.

\begin{proposition}\label{induc other case}
For every $V \in \Prog_{n-1,n}$, 
\[
\P_{L_{1,n}} (T_V < \infty) > 0.
\]
\end{proposition}

\begin{proof}
If $V \in \Prog_{n-1,n}$, then $V$ is connected. Hence, if $L_{1,\nn}$ intersects $V$, then CAT can form $V$ in one step. Otherwise, we can treadmill the elements of $L_{1,\nn}$ until they intersect $V$.
\end{proof}

The next result addresses the second case.

\begin{proposition}\label{base of set to line induction}
For any $V \in \Prog_{1,\nn}$,
\[
\P_{L_{1,\nn}} \big( T_V < \infty \big) > 0.
\]
\end{proposition}

We omit a proof of the simple fact that if CAT can form $B$ from $A$, then CAT can form $B \cup \{a\}$ from $A \cup \{a\}$ for $a \in \Z^d$, so long as $\|a\|$ is sufficiently large.

\begin{lemma}\label{no interference lemma}
Let $A, B \in \cal S$. If $a \in \Z^\dd$ has sufficiently large $\|a\|$, then
\begin{equation}\label{no interference}
\P_A (T_B < \infty) > 0 \implies \P_{A \cup \{a\}} (T_{B \, \cup \, \{a\}} < \infty) > 0.
\end{equation}
\end{lemma}

\begin{proof}[Proof of Proposition~\ref{base of set to line induction}]
We use induction on $\nn \geq 2$. Proposition~\ref{induc other case} addresses the base case of $\nn=2$.  
Accordingly, we assume that 
\begin{equation}\label{base ih}
\forall U \in \Prog_{1,n}, \quad \P_{L_{1,\nn}} (T_U < \infty) > 0,
\end{equation}
and we aim to show that
\begin{equation}\label{base ih aim}
\forall V \in \Prog_{1,n+1}, \quad \P_{L_{1,\nn+1}} (T_V < \infty) > 0.
\end{equation}

We will isolate $(\nn+1) e_1 \in L_{1,\nn+1}$ far from the origin. The induction hypothesis and Lemma~\ref{no interference lemma} will then allow us to form a set $W$, which is nearly $V$, from $L_{1,\nn}$. Specifically, we will use the induction hypothesis to form 
\[
     W = (V \setminus \{v,g\}) \cup \{\ell-e_1\},
\]
where $v$ is any non-isolated element of $V$, and where $g$ and $\ell$ are the greatest and least elements of $V$ in the lexicographic order. We will then change $W$ into $V$ in two subsequent steps.

We can form $V$ as follows.
\begin{enumerate}
    \item[(1)] {\em Isolate an element}. Starting from $L_{1,\nn+1}$, we treadmill the pair $\{\nn e_1, (\nn+1) e_1\}$ in the $e_1$ direction, until one of them, at say, $a$, has norm $\|a\|$ which is sufficiently large to ensure that \eqref{no interference} holds with $A = L_{1,\nn}$ and $B = W$. We then return the other element to $\nn e_1$ to form $L_{1,\nn} \cup \{a\}$.
    \item[(2)] {\em Use the induction hypothesis}. By design, $W \in \Prog_{1,n}$ because $W$ contains $\{\ell,\ell-e_1\}$. Hence, the induction hypothesis implies that we can form $W$ from $L_\nn$. By Lemma~\ref{no interference lemma}, we can in fact form $W \cup \{a\}$ from $L_{1,\nn} \cup \{a\}$. 
    \item[(3)] {\em Return the isolated element}. From $W \cup \{a\}$, we activate at $a$ and transport to $\ell-2e_1 \in \bdy W$.
    \item[(4)] {\em Treadmill the lex elements}. We then treadmill the pair $\{\ell-e_1,\ell-2e_1\}$ to $\{g,g+e_d\}$, which is possible because $\ell- 2e_1$ and $g$ are connected by a path which lies outside of $W$. The result is $(V \setminus \{v\}) \cup \{g+e_d\}$. Lastly, we activate at $g+e_d$ and transport to $v$, which is possible because $v$ is non-isolated in $V$.
\end{enumerate}
\end{proof}

The induction step in the proof of Proposition~\ref{line to set} is easier than that of Proposition~\ref{base of set to line induction}. The reason is that, for some $U \in \Prog_{1,\nn+1}$, there is no non-isolated $u \in U$ such that $U \setminus \{u\} \in \Prog_{1,n}$. However, for every $V \in \Prog_{m+1,\nn+1}$, there is a non-isolated $v \in V$ for which $V \setminus \{v\} \in \Prog_{m,n}$. (Note that, in the latter case, both $\nn+1$ and $\kk+1$ are decremented to $\nn$ and $\kk$.) This follows from the definition of $\Prog_{m+1,n+1}$, which requires that each of its members $V$ has a $(\kk+1)$-element subset $\{y_1,\dots,y_{\kk+1}\}$ which, in particular, satisfies
\begin{equation}\label{induc key obs}
y_{\kk+1} \in \bdy (V \setminus \{y_{\kk+1}\}) \quad \text{and} \quad V \setminus \{y_{\kk+1}\} \in \Prog_{m,n}.
\end{equation}
In other words, $y_{\kk+1}$ is the aforementioned $v$.

\begin{proof}[Proof of Proposition~\ref{line to set}]
We use induction on $\kk \geq 1$. The base case of $\kk = 1$ holds by Proposition~\ref{base of set to line induction}. Accordingly, we assume that, for all $n \geq \kk + 1$, 
\begin{equation}\label{induc ih}
\forall U \in \Prog_{m,n}, \quad \P_{L_{1,n}} ( T_U < \infty) > 0.
\end{equation}
We aim to prove that, for all $n' \geq \kk+2$,
\begin{equation}\label{induc aim}
\forall V \in \Prog_{m+1,n'}, \quad \P_{L_{1,n'}} ( T_V < \infty) > 0,
\end{equation}
In fact, we only need to address $n' \geq \kk+3$, because Proposition~\ref{induc other case} proves \eqref{induc aim} when $n' = \kk+2$.

Assume that $n' \geq \kk+3$. We can form $V$ as follows.
\begin{enumerate}
\item[(1)] {\em Isolate an element}. Starting from $L_{1,n'}$, we treadmill the pair $\{(n'-1)e_1, n'e_1\}$ in the $e_1$ direction, until one of them, at say, $a$, has a radius which is sufficiently large to ensure that \eqref{no interference} holds with $A = L_{1,n'}$ and $B = V \setminus \{y_{\kk+1}\}$. We then return the other element to $n' e_1$ to form $L_{1,n'-1} \cup \{a\}$.

Note that, to treadmill the pair without moving the other elements, we use the fact that $n' \geq \kk+3$. Specifically, with each step, we activate one element of the pair and $\kk$ of the other elements. Since there are at least $\kk+1$ elements which are not of the pair, we can transport the $\kk$ elements we activate to the sites at which they were activated.
\item[(2)] {\em Use the induction hypothesis}. By \eqref{induc key obs}, $V \setminus \{y_{\kk+1}\} \in \Prog_{m,n'-1}$. Hence, we can apply the induction hypothesis \eqref{induc ih} with $n = n'-1$ and $U = V \setminus \{y_{\kk+1}\}$ to form $U$ from $L_{1,n'-1}$. In fact, by Lemma~\ref{no interference lemma}, we can form $U \cup \{a\}$ from $L_{1,n'-1} \cup \{a\}$.
\item[(3)] {\em Return the isolated element}. Lastly, activate at $\{y_1,\dots,y_{\kk-1}\}$ and $w$, and transport to $\{y_1,\dots,y_\kk\}$ to form $V$.
\end{enumerate}
\end{proof}


\newcommand{\etalchar}[1]{$^{#1}$}

\end{document}